\documentclass[11pt]{amsart}
\usepackage{amsmath, amsthm, amscd, amssymb, amsfonts, amsxtra, amssymb, latexsym}
\usepackage{enumerate}
\usepackage{verbatim}
\usepackage{float}
\usepackage{graphicx,epsfig,tikz}
\usepackage{hyperref}
\usepackage{marginnote}
\hypersetup{colorlinks = true,	allcolors  = blue}

\newcommand{\na}{\mathbb{N}}
\newcommand{\z}{\mathbb{Z}}

\newcommand{\sg}{\mathbb{S}}

\newcommand{\rr}{\mathcal{L}}
\newcommand{\ff}{\mathbb{F}}
\newcommand{\pl}{\mathbb{P}}

\newcommand{\cod}{\mathcal{C}}

\newcommand{\sop}{\operatorname{Sup}}

\newcommand{\aut}{\operatorname{Aut}}
\newcommand{\gal}{\operatorname{Gal}}
\newcommand{\divisor}{\operatorname{Div}}

\newcommand{\Char}{\mathrm{char}}
\newcommand{\pgl}{\mathrm{PGL}}

\newcommand{\sk}{\smallskip}
\newcommand{\msk}{\medskip}
\newcommand{\bsk}{\bigskip}

\newtheorem{thm}{Theorem}[section]
\newtheorem{prop}[thm]{Proposition}
\newtheorem{lem}[thm]{Lemma}
\newtheorem{coro}[thm]{Corollary}
\theoremstyle{definition}
\newtheorem{rem}[thm]{Remark}
\newtheorem{exam}[thm]{Example}
\newtheorem{defi}[thm]{Definition}

\theoremstyle{remark}

\begin{document}
\numberwithin{equation}{section}
\title[Cyclic AG-codes]{On cyclic algebraic-geometry codes}
\author[G. Caba\~na, M. Chara, R. Podest\'a, R. Toledano]{Gustavo Caba\~na, Mar\'ia Chara, Ricardo Podest\'a, Ricardo Toledano}
\dedicatory{\today}
\keywords{Cyclic codes, AG-codes, algebraic function fields, asymptotic goodness, towers.}
\thanks{2010 {\it Mathematics Subject Classification.} Primary 94B27;\,Secondary 94B15, 20B25.}
\thanks{Partially supported by CONICET, FONCyT, SECyT-UNC, UNL CAI+D 2020}
\address{Gustavo Caba\~na -- UNL - CONICET, (3000) Santa Fe, Argentina. {\it E-mail: cabanagusti@gmail.com}}
\address{Mar\'ia Chara -- UNL - CONICET, (3000) Santa Fe, Argentina. {\it E-mail: mchara@santafe-conicet.gov.ar}}
\address{Ricardo Podest\'a -- CIEM-CONICET, FaMAF, Universidad Nacional de C\'ordoba, Av.\@ Medina Allende 2144, Ciudad Universitaria, (5000) C\'ordoba, Argentina. {\it E-mail: podesta@famaf.unc.edu.ar}}
\address{Ricardo Toledano -- UNL, Departamento de Matem\'atica, Facultad de Ingenier\'ia Qu\'imica, Santiago del Estero 2829, (3000) Santa Fe, Argentina.  {\it E-mail: ridatole@gmail.com}}

\begin{abstract}
In this paper we initiate the study of cyclic algebraic geometry codes. We give conditions to construct cyclic algebraic geometry codes in the context of algebraic function fields over a finite field by using their group of automorphisms. 
We  prove that cyclic algebraic geometry codes constructed in this way are closely related to cyclic extensions. We also give a detailed study of the monomial equivalence of cyclic algebraic geometry codes constructed with our method in the case of a rational function field. 
\end{abstract}

\maketitle

\section{Introduction}
\subsubsection*{Motivation}
Let $\ff_q$ be a finite field of $q$ elements.
A linear code of length $n$ and dimension $k$   over $\ff_q$ is simply an $\ff_q$-linear subspace $\cod $ of $\ff_q^n$ of dimension $\dim \cod = k$. In this case it is customary to say that $\cod$ is an $[n,k]$-code (over $\ff_q$) and the elements of $\cod$ are called codewords. An important parameter to consider in an $[n,k]$-code $\cod$ is its minimum Hamming distance
$d(\cod)=\min\{ w(c) : 0\neq c\in \cod\}$,
where $w(c)$ is the weight of $c\in \cod$, the number of non zero coordinates of $c$. An $[n,k,d]$-code $\cod$ is just an $[n,k]$-code $\cod$ such that $d=d(\cod)$.

Among the classical linear codes over $\ff_q$, the family of cyclic codes have shown to be one of the most important and widely used because of their good parameters, excellent detection-correction capabilities and fast and efficient encoding-decoding algorithms. Let us recall that a code $\cod$ is \textit{cyclic} if
it is closed under cyclic permutations of the coordinates of its codewords. That is,
for any $c=(c_1,c_2,\ldots,c_n)\in \cod$, the cyclic shift
{$s(c)=s(c_1,c_2, \ldots,c_n) = (c_2,\ldots, c_n, c_1)$} is also in $\cod$.
There is a natural action of the symmetric group  $\sg_n$ 
on $\ff_q^n$ defined as
$\tau (a_1,\ldots,a_n) = (a_{\tau(1)},\ldots,a_{\tau(n)})$
for $(a_1,\ldots,a_n)\in \ff_q^n$, $\tau\in \sg_n$.
This action  defines the so called permutation  automorphism group $\mathrm{PAut}(\cod)$ of $\cod$ as the subgroup of $\sg_n$ preserving $\cod$, that is
$$\mathrm{PAut}(\cod) = \{\tau\in \sg_n : \tau (\cod) = \cod\}.$$
We see at once that a code $\cod$ is cyclic if and only if the $n$-cycle $\sigma=(12 \cdots n) \in \mathrm{PAut}(\cod)$. Clearly the above mentioned  cyclic shift $s$ corresponds to the $n$-cycle $\sigma=(12 \cdots n)$. It is worth to mention that the family of cyclic codes contains  important codes  such as Golay codes, binary Hamming codes, Reed-Solomon codes and BCH codes.

A major breakthrough in coding theory was given by Goppa at the beginning of the 80's 
when he introduced a whole new family of linear codes obtained by evaluation of rational functions on rational points of an irreducible and smooth projective  curve over $\ff_q$. These codes are known today as algebraic geometry codes (or AG-codes for short). Using  Goppa's ideas and modular curves over $\ff_q$, Tsfasmann et al (\cite{TVZ82}) constructed a family of codes which surpassed the Gilbert-Varshamov bound for the very first time. Thus, it seems natural to consider cyclic algebraic geometry codes because, in this way, we will have the above mentioned advantages of cyclic codes combined with the conceptual richness  involved in the construction of AG-codes.

Following the book \cite{Sti}, we will use the language of function fields to describe the construction of AG-codes.  Let $F$ be a function field over $\ff_q$, that is, $F$ is a finite field extension of a rational function field $\ff_q(x)$. Let $D=P_1+\cdots +P_n$ and $G$ be disjoint divisors of $F$, where $P_1,\ldots,P_n$ are different rational (degree one) places of $F$.
The AG-code defined by $F$, $D$ and $G$ is
\begin{equation}\label{C(DG)}
 C_{\rr}(D,G) = \big \{ \big( z(P_1),z(P_2),\ldots,z(P_n) \big) \in \ff_q^n : z\in \rr(G) \big\},
\end{equation}
where $z(P_i)$ stands for the residue class of $z$ modulo $P_i$ and
\begin{equation} \label{lG}
\rr(G) = \{z\in F^* : (z) \ge -G\} \cup \{0\},
\end{equation}
is the Riemann-Roch space associated to $G$, with $F^*=F\smallsetminus \{0\}$.
 
If we want to emphasize the dependence of $C_{\rr}(D,G)$ on $F$, we will say that $C_{\rr}(D,G)$ is \textit{defined over} $F$. It is also customary to say that an AG-code $C_{\rr}(D,G)$ is an  $n$-point AG-code when the support of the divisor $G$ has $n$ different places of $F$. Some other standard references for AG-codes are the books \cite{MMR}, \cite{Mo}, \cite{Ste} and \cite{TVN}.

Using a clever idea, Pellikaan et al (\cite{PSW}) proved that any linear code $\cod$, in particular any cyclic code, can be represented as an AG-code as in \eqref{C(DG)}, although the proof is not constructive. That is, there is an algebraic function field $F$ over $\ff_q$ 
and disjoint divisors $D,G$ of $F$ such that $\cod = C_{\mathcal{L}}(D,G)$.
In this way, there is an interesting relation between the classical algebraic construction of codes and the new geometric one. See for instance \cite{BHHW}, for a brief survey on the historical transition from classical algebraic codes to AG-codes. On the other hand, Stichtenoth has shown that every cyclic code can be  realized as the trace code of generalized AG-codes defined over rational function fields (see Section 9.2 in \cite{Sti} for details). However it is not clear how to construct  cyclic AG-codes in a systematic way. 

Once a construction method for cyclic AG-codes is found, an important question to answer is the following one: how many inequivalent (in the sense of monomial equivalence of linear codes as given in Definition \ref{defi}) cyclic AG-codes can be constructed?  Answering this last question may have interesting consequences in the problem of constructing sequences of cyclic AG-codes by using asymptotically good towers of function fields. Let us recall that the question of whether or not the family of cyclic codes is asymptotically good is still open and one way of constructing asymptotically good sequences of codes over $\ff_q$ is by using asymptotically good towers of function fields over $\ff_q$ (see for instance \cite{CPT20} and \cite{Sti06}).

When dealing with the construction of cyclic AG-codes the following basic question arises: what does the condition of cyclicity look like for an AG-code?  By definition,
$C_{\rr}(D,G)$ is cyclic if 
\[s(c) = \big(u(P_2),\ldots,u(P_{n}),u(P_{1})\big) \in C_{\rr}(D,G),\] for any $c=(u(P_1),\ldots,u(P_n))\in C_{\rr}(D,G)$ where $u\in \rr(G)$. This happens if and only if there exists $v\in \rr(G)$ such that
$s(c) = (v(P_1), v(P_2),\ldots, v(P_n))$. Hence, $C_{\rr}(D,G)$ is cyclic if for every $u \in \rr(G)$ there exists $v\in \rr(G)$ such that $v(P_i)=u(P_{i+1\mod n})$ for $1\leq i\leq n$, that is
\begin{equation} \label{AGcyc conds}
\left\{
\begin{array}{rcl}
v(P_1) &=& u(P_2), \\
v(P_2) &=& u(P_3),\\ 
&\vdots& \\
v(P_{n-1})&=& u(P_n) , \\
v(P_n)&=& u(P_{1}).
\end{array} \right.
\end{equation}

Thus the construction problem of  cyclic AG-codes boils down to answer the following questions: how and when do we find  an element $v\in \rr(G)$ satisfying \eqref{AGcyc conds}? One way of looking for an element $v\in \rr(G)$ solving \eqref{AGcyc conds} is by considering the group $\aut_{\ff_q}(F)$ of $\ff_q$-automorphisms of the function field $F$, that is the group of automorphisms of $F$ fixing $\ff_q$ pointwise (see Section \ref{sigma}).

The main goal of this work is to initiate the study of cyclic AG-codes which are constructed by using the group $\aut_{\ff_q}(F)$. In particular, we will study in detail the problem of monomial equivalence of    cyclic AG-codes over a rational function field $\ff_q(x)$ which are constructed by using the group $\aut_{\ff_q}(\ff_q(x))$ (see Section \ref{sigmarational}).  

\subsubsection*{Outline and results}
A brief summary of the paper is as follows.
In Section 2, we recall some basic facts about AG-codes defined over a function field $F$ over $\ff_q$ and its $\ff_q$-automorphism group $\aut_{\ff_q}(F)$. We then present a method, which will be called the {\it sigma-method}, to construct cyclic AG-codes $C_\rr(D,G)$, based on the action of $\sigma \in \aut_{\ff_q}(F)$ on the rational places in the support of $D$.  

In Section 3 we present some examples of cyclic AG-codes using the sigma-method in cyclic extensions of function fields. Section 4  is devoted to the construction of what we call sigma-cyclic rational codes, that is  cyclic AG-codes over the rational function field $F=\ff_q(x)$ defined by the sigma-method with the group $\aut_{\ff_q}(\ff_q(x))$.  We also consider the question of the monomial equivalence of sigma-cyclic rational codes of the form $C_\rr(D,rP)$ where $P$ is a rational place of $\ff_q(x)$. As an interesting consequence of this study, we prove (see Theorem \ref{mainequality}) one of the main results of this work: for a given length and dimension there is, up to monomial equivalence, only one sigma-cyclic rational code over $\ff_q(x)$ which can be chosen to be of the form $C_\rr(D,rP_\infty)$, where $P_\infty$ denotes the only pole of $x$ in $\ff_q(x)$.

Finally in Section 5 we study  the structural properties of the sigma-method in an arbitrary function field $F$ over  $\ff_q$. The main result here is given in Theorem \ref{ext cic} where we prove that the sigma-method with the group $\aut_{\ff_q}(F)$ is related to the existence of a subfield $E$ of $F$ such that $F/E$ is a cyclic extension. A precise description of the ramification of the involved places is also given.

\section{Cyclicity condition for AG-codes and automorphisms}  \label{sigma}
In this section we first recall some basic facts on AG-codes  over a function field $F$, the automorphism group of an AG-code and the group of automorphisms of $F$. Then we consider the cyclicity condition \eqref{AGcyc conds} for AG-codes in terms of automorphisms of $F$.

\subsubsection*{AG-codes}
Throughout this work $\ff_q$ will always be algebraically closed in any considered function field $F$ over $\ff_q$, i.e.\@ $\ff_q$ is the full constant field of $F$. 
  These type of extensions of function fields are called {\em geometric}.

From now on, we will denote by $\mathbb{P}(F)$ the set of places of a function field $F$ over $\ff_q$ and by $\mathbb{P}_1(F)$ the subset of rational places. Also, $\divisor (F)$ stands for the group of divisors of $F$ and $\sop(D)$ for the support
of $D\in \divisor (F)$.

\sk
Consider an AG-code $C_\rr(D,G)$ defined over a  function field $F$ over $\ff_q$ of genus $g$ as in \eqref{C(DG)}.
Denote by $N(F) = \#\{P\in \mathbb{P}(F) : \deg P=1\} = \# \mathbb{P}_1(F)$
the number of rational places of $F$ as usual. It is well known that $C_\rr(D,G)$ is an $[n,k,d]$-code with
$$n\le N(F) \le q+1+g[2\sqrt q],$$
by Serre's improvement of the Hasse-Weil bound,
\begin{equation} \label{dist}
d \ge  n-\deg G,
\end{equation}
and $k=\ell(G)-\ell(G-D)$ where
$\ell(G) =\dim_{\ff_q} \rr(G)$. 
If $\deg G <\deg D=n$, then \linebreak $\rr(G-D)=0$ and hence, by the Riemann-Roch theorem, we have
\begin{equation} \label{kG}
k=\ell(G) \ge \deg G + 1 -g,
\end{equation}
with equality if $\deg G \ge 2g-1$. This and the Singleton bound imply 
$$n+1-g\le k+d \le n+1.$$

\subsubsection*{Automorphisms and permutation automorphisms}
Let $F$ be a function field over $\ff_q$. If $E$ is a function field extension of $F$, we will denote as usual by $Q\,|\,P$ the fact that $Q$ is a place in $E$ over a place $P$ in $F$, that is $P=Q\cap F$.
 
 We have the following basic result about places in finite extensions of function fields and isomorphisms that can be proved using similar arguments to the ones given in the Galois case (see for instance Lemma 3.5.2 and Theorem 3.7.1 in \cite{Sti}).

\begin{lem} \label{lema sigma}
	   Let $F$ be a  function field over $\ff_q$ and let $F'$ and $E'$ be extensions of $F$. Suppose 
  $\sigma:F'\rightarrow E'$ is an isomorphism and put $E=\sigma(F)$.
Then the following holds.

\noindent $(a)$  If $Q \in \mathbb{P}(F')$  then $\sigma(Q) = \{ \sigma(z) : z \in Q\} \in \mathbb{P}(E')$
		         and $\sigma(\mathcal{O}_Q) = \mathcal{O}_{\sigma(Q)}$.

\noindent $(b)$  If  $0\neq z'\in E'$ then $v_{\sigma(Q)}(z')=v_Q(\sigma^{-1}(z'))$.

\noindent $(c)$  Let $Q \in \mathbb{P}(F')$ and $P \in \mathbb{P}(F)$ such that $Q \,|\, P$. 
Then, $\sigma(Q) \,|\, \sigma(P)$ and also 
$$e(\sigma(Q) \,|\, \sigma(P)) = e(Q \,|\, P) \qquad \text{and} \qquad 
f(\sigma(Q) \, | \, \sigma(P))=f(Q \, | \, P),$$ 
where $e$ is the ramification index and $f$ is the inertia degree.
		
\noindent $(d)$ The action of $\sigma:F' \rightarrow E'$ on $\pl(F')$ is naturally extended to an action on $\divisor(F')$ by linearity. 

\noindent $(e)$ If $F'/F$ is Galois, then $G=\gal(F'/F)$ acts transitively on the set of places of $F'$, i.e.\@ for each pair of places $P,Q \in \pl(F')$ such that $P\cap F=Q\cap F$
there exists an element $\sigma\in G$ such that $\sigma(P)=Q$. 
\end{lem}

It is a basic result that $\aut_{\ff_q}(F)$ is a finite group.  Now from ($a$) of Lemma \ref{lema sigma}, we have that $\aut_{\ff_q}(F)$ acts naturally on the group $\mathrm{Div}(F)$ of divisors of $F$ by defining,
\[\sigma \big( \sum_P a_{P}P \big)  = \sum_P a_P \, \sigma(P),\]
for each $\sigma\in \aut_{\ff_q}(F)$. 
Thus, for any AG-code $\cod=C_{\rr}(D,G)$ defined over $F$ with
$D=P_1+\cdots+P_n$, where each $P_i$ is a rational place of $F$,
and for any $\sigma \in \aut_{\ff_q}(F)$
we have a well defined AG-code $\cod^\sigma =C_{\rr}(\sigma(D),\sigma(G))$.

Now, suppose that there is an element $\sigma \in  \aut_{\ff_q}(F)$ fixing $D$ and $G$, i.e.\@
\begin{equation}\label{invariantag}
\sigma(D)=D \qquad \text{and} \qquad \sigma(G)=G. 
\end{equation}
Hence, $\cod^\sigma=\cod$.
Condition $\sigma(G)=G$ guarantees that  
\begin{equation}\label{slG}
z\in \rr(G) \quad \Leftrightarrow \quad \sigma^{-1}(z)\in\rr(G).
\end{equation}
The direct implication follows directly from ($b$) of Lemma \ref{lema sigma} and the fact that $Q\in \sop(G)$ if and only if  $\sigma(Q)\in\sop(G)$. The converse implication follows from the fact that $\aut_{\ff_q}(F)$ is of finite order, say $m$,  so that $\sigma^{-m}(z)=z\in \rr(G)$.

From \eqref{slG} we see that if $\sigma\in \aut_{\ff_q}(F)$ satisfies \eqref{invariantag} and $(z(P_1),\ldots,z(P_n))$ is a codeword of $C_{\rr}(D,G)$, we get another codeword defined as 
$$\sigma \cdot \big( z(P_1),\ldots,z(P_n) \big) = \big( \big(\sigma^{-1}(z)\big)(P_1),\ldots,\big(\sigma^{-1}(z)\big)(P_n) \big).$$
The map  {$z(\sigma(P))\mapsto (\sigma^{-1}(z))(P)$} defines a field isomorphism between the residue fields
$\mathcal{O}_{\sigma(P)}/\sigma(P)$ and $\mathcal{O}_P/P$ for any $P\in \mathbb{P}(F)$ and $\sigma \in \aut_{\ff_q}(F)$. 
Hence we can consider that 
 {\begin{equation} \label{xP}
z(\sigma(P))=(\sigma^{-1}(z))(P). 
\end{equation}}
Furthermore, condition $\sigma(D)=D$ implies that the codeword
 $\sigma \cdot \big( z(P_1),\ldots,z(P_n) \big) $
represents a permutation of the coordinates of the codeword $(z(P_1),\ldots,z(P_n))$ of $C_{\rr}(D,G)$ for any $\sigma\in \aut_{\ff_q}(F)$.  

Thus, if $\sigma\in\aut_{\ff_q}(F)$ satisfies \eqref{invariantag} then we can think of 
$\sigma$ as an element of the permutation automorphism group $\mathrm{PAut}(C_{\rr}(D,G))$. 
In view of these observations it is natural to consider the group
\begin{equation} \label{autC}
\aut_{D,G}(F) = \{ \sigma\in \aut_{\ff_q}(F) \,:\, \sigma(D)=D \,\text{ and } \, \sigma(G)=G \}.
\end{equation}
It is shown in Proposition 8.2.3 of \cite{Sti} that if
 $n>2g+2$ then $\aut_{D,G}(F)$ can be viewed as a subgroup of $\mathrm{PAut}(C_\rr(D,G))$.

\subsubsection*{Automorphisms and cyclic AG-codes}
The next result will allow us to formulate a procedure to construct cyclic AG-codes. This procedure will be called the sigma-method and every cyclic AG-code considered in this work will be constructed with the sigma-method.

\begin{lem} \label{ej cic}
Let $P_1,\ldots,P_n$ be $n$ different rational places of a function field $F$ over $\ff_q$ and  $G$ be divisor of $F$ with disjoint support with respect to   $D=P_1+\cdots +P_n$. Suppose that there exists $\sigma \in \aut_{D,G}(F)$  such that 
 {\begin{equation} \label{conditions}
\sigma(P_1) = P_2,\: \ldots,\: \sigma(P_{n-1}) = P_{n}, \: \sigma(P_n) = P_1.
\end{equation}}
Then $C_{\rr}(D,G)$ is a cyclic AG-code, the order of $\sigma$ as an element of $\aut_{\ff_q}(F)$ is divisible by $n$ and also $n$ is the smallest positive integer satisfying $\sigma^n(P_1)=P_1$.
\end{lem}

\begin{proof}
We recall from \eqref{AGcyc conds}   
that $C_{\rr}(D,G)$ is cyclic if and only if for each $u\in \rr(G)$, there exists an element $v\in\rr(G)$ such that
\[v(P_i) = u(P_{i+1 \text{ mod } n}) \qquad 1\le i \le n.\]
Suppose now that $\sigma\in \aut_{D,G}(F)$ satisfies \eqref{conditions}. For each  
$u\in \rr(G)$ the element $v=\sigma^{-1}(u)$ belongs to $\rr(G)$ by \eqref{slG} and satisfies \eqref{AGcyc conds}. In fact, by \eqref{xP},
$$v(P_i) = \big(\sigma^{-1}(u)\big) (P_i) = u(\sigma(P_i)) = u(P_{i+1 \text{ mod } n})$$
for each $1\le i \le n$ and then $C_{\rr}(D,G)$ is a cyclic AG-code.

Now, let $m$ be the order of $\sigma$ in $\aut_{\ff_q}(F)$. Then $\sigma^m=id$ and thus we have that $\sigma^m(P)=P$ for any place $P$ of $F$. On the other hand, notice that 
\[\begin{array}{rcl}
	P_2 &=& \sigma(P_1),\\
	P_3 &=& \sigma^2(P_1),\\
 		& \vdots & \\
	P_n &=&\sigma^{n-1}(P_1),\\
	P_1 &=&\sigma^n(P_1). 
\end{array}\]
In particular we see that $\sigma^{nk}(P_1)=P_1$ for any $k\in \na$. If $m<n$ then $P_1=\sigma^m(P_1)=P_{m+1}$, so that $P_1\in \{P_2,\ldots,P_n\}$ contradicting the assumption that the places $P_1,\ldots,P_n$ are $n$ different places. Thus $m\geq n$ and then there exist unique integers $k\geq 1$ and $r\geq 0$ such that $m=kn+r$ where $r=0$ or $1\leq r\leq n-1$. If $r\neq 0$ then $1\leq r\leq n-1$ and so
\[P_1=\sigma^m(P_1)=\sigma^{r+kn}(P_1)=\sigma^r(\sigma^{nk}(P_1))=\sigma^r(P_1)=P_{r+1},\] 
so that $P_1\in \{P_2,\ldots,P_n\}$ which is a contradiction as we have already noted. Therefore $r=0$ and so we must have $m=kn$. Finally, it is clear from the above argument that we can not have $\sigma^j(P_1)=P_1$ for some positive integer $j<n$. 
\end{proof}

Roughly speaking, what Lemma \ref{ej cic} is saying that a cyclic shift of the places in $\sop (D) = \{ P_1,\ldots, P_n \}$ by some suitable element of $\aut_{D,G}(F)$ ensures that $C_{\rr}(D,G)$ is a cyclic AG-code.  

\begin{rem} \label{relabeling}
Clearly \eqref{conditions} represents a choice on the numbering of the indices of the $n$ rational places. Different numberings for the indices of the places $P_1,\ldots,P_n$ will give equivalent codes in the sense of the definition below.
\end{rem}

\begin{defi} \label{defi}
Two codes $\cod_1$ and $\cod_2$ over $\ff_q$ are called \emph{equivalent}, and we write $\cod_1\sim \cod_2$, if there exists a monomial matrix $M$ over $\ff_q$ such that $\cod_2=\cod_1 M$ (a monomial matrix over $\ff_q$ is a matrix such that in each row and column there is only one non zero element of $\ff_q$). In other words, $\cod_1\sim \cod_2$ if each codeword of $\cod_2$  can be obtained from the codewords of $\cod_1$ by a combination  of the following two operations:
($a$) permutation of the digits of a codeword;
($b$) multiplication of each entry of a codeword by a non-zero element of $\ff_q$ (not necessarily the same element for each entry).
\end{defi}

In view of Lemma \ref{ej cic} we have the following definition:
\begin{defi} 
Let $\cod=C_{\rr}(D,G)$ be an AG-code defined over a function field $F$ over $\ff_q$ with $D=P_1+\cdots+P_n$. We shall say that $\cod$ is \textit{sigma-cyclic} if it is cyclic and the permutation of the codewords is performed by an automorphism fixing $D$ and $G$, while permuting the places $P_i$; that is, if there exists an automorphism $\sigma \in \aut_{D,G}(F)$ such that \eqref{conditions} holds.  
\end{defi}

\subsubsection*{The sigma-method} 
We express now the condition \eqref{conditions} of Lemma \ref{ej cic} in a more structural way. Let $F$ be a function field over $\ff_q$ and let $\sigma\in \aut_{\ff_q}(F)$ be of order $m$. For a given place $P$ of $F$ we denote by $[P]_\sigma$ the orbit defined by the action of the cyclic subgroup $\langle \sigma\rangle$ of $\aut_{\ff_q}(F)$ generated by $\sigma$ on the set of places of $F$, that is 
$$ [P]_\sigma = \{\sigma(P), \sigma^2(P), \ldots, \sigma^m(P)=P \}.$$
If we consider the $n$ rational places $P_1,\ldots,P_n$ of Lemma \ref{ej cic} we see from its proof that $[P_1]_\sigma=\{P_1,\ldots,P_n\}$ where $\sigma\in \aut_{D,G}(F)$ satisfies \eqref{conditions}. Therefore we have that $m=nk$ where $k$ is the order of the isotropy group 
$$\langle\sigma\rangle_{P_1}=\{\sigma^i:\sigma^i(P_1)=P_1\}.$$ 
Notice also that from the proof of Lemma \ref{ej cic}, this group can be described more explicitly as $\langle\sigma\rangle_{P_1}=\{\sigma^{ik}\}_{i=1}^m$.

We describe now what we call the \textit{sigma-method} to construct sigma-cyclic AG-codes over a function field $F$ over $\ff_q$ which is, essentially, a reformulation of Lemma \ref{ej cic} in terms of orbits. 

\msk 

\centerline{\textsc{The sigma-method}} \msk \hrule
\begin{enumerate}[($a$)]
	\item Find $\sigma \in \aut_{\ff_q}(F)$ of order $m\geq 2$  and a divisor $G$ of $F$ such that $\sigma(G)=G$. \sk 
	
	\item Find a rational place $P$ of $F$ such that $P\notin\sop D$ and $\sigma(P)\neq P$ (if $\sigma(P)=P$ the orbit $[P]_\sigma$ is trivial, namely $[P]_\sigma=\{P\}$). \sk 
	
	\item Find the order $k$ of the isotropy group $\langle\sigma\rangle_P$.  \sk 

	\item  
	Let $n=m/k$. We have that $n$ is the smallest divisor of $m$ such $\sigma^n(P)=P$. Then  the places $P,\sigma(P),\sigma^2(P),\ldots,\sigma^{n-1}(P)$ are $n$ different rational places of $F$ and 
\begin{equation}\label{orbitsigmamethod}
	[P]_\sigma = \{P,\sigma(P),\sigma^2(P),\ldots,\sigma^{n-1}(P)\}.
\end{equation}

	\item If we define $D = P+\sigma(P)+\cdots+\sigma^{n-1}(P)$ we have that $D$ and $G$ are disjoint divisors. 
	
	\item $C_{\rr}(D,G)$ is a sigma-cyclic AG-code over $\ff_q$, because if we write $P_1=P$ and $P_{i+1}=\sigma^i(P)$ for $i=1,\ldots,n-1,$ then it is easy to check that $D=P_1+P_2+\cdots+P_n$ so that $D$ is a divisor of $F$ satisfying the conditions of Lemma~\ref{ej cic}.
\end{enumerate}
\hrule

\bsk

\begin{rem}
	The length of a sigma-cyclic code constructed with the sigma-method depends on the size of the orbit \eqref{orbitsigmamethod}. The determination of the size of the orbit \eqref{orbitsigmamethod}, which is equivalent to find either the order of the isotropy group $\langle\sigma\rangle_P$ or to  prove that $n$ is the smallest  divisor of $m$ such $\sigma^n(P)=P$, is one of the main difficulties to overcome when using this method.
\end{rem}

One of the most basic and favorable settings to use the sigma-method is the rational function field $\ff_q(x)$. This is so because not only the $\ff_q$-automorphism group of $\ff_q(x)$ is well known, but also because the rational places of $\ff_q(x)$ have a simple description and from \cite{LN99} we know the size of the orbits. These advantages will be fully exploited in Section~4.     
 
For the time being we can give a more explicit version of Lemma \ref{ej cic} in the case of a rational function field. This version will allow us to give some simple examples of sigma-cyclic AG-codes. We first fix some notation that will be used from now on. We will write $P_{\alpha}$ (resp.\@ $P_{\infty}$) to denote the place of $\ff_q(x)$ which is the only zero (resp.\@ pole) of $x-\alpha$ (resp.\@ $x$) in $\ff_q(x)$. It is well known (see, for instance, \cite{Sti})  that the places $P_\alpha$ with $\alpha\in \ff_q$ and $P_{\infty}$ are all the rational places of $\ff_q(x)$.

\begin{lem} \label{lemin ideals}
Let $\alpha_1,\ldots,\alpha_n,\beta \in \ff_q$ be all different. 
Suppose that  there exists an element $\sigma \in \aut(\ff_q(x))$  such that
\begin{equation} \label{cond ideals}
			\sigma(x-\alpha_i) \in P_{\alpha_{i+1}} \text{ for $i$ mod $n$}, \qquad 
			\sigma(x-\beta) \in P_{\beta}, \qquad \text{and} \qquad \sigma(x^{-1}) \in P_{\infty}.
	\end{equation} 
Then, the code $C_\rr(D,G)$ with 
	$$D=P_{\alpha_1}+\cdots+P_{\alpha_n} \qquad \text{ and } \qquad G=rP_\beta+sP_\infty, \quad r,s \in \z,$$ 
	is a sigma-cyclic AG-code of length $n$ over $\ff_q$. 
	Also, if $0< r + s <n$ then $C_{\rr}(D,G)$ is a non-trivial MDS code with $k=r+s+1$ and $d=n-(r+s)$.
\end{lem}

\begin{proof}
It is clear that $D$ and $G$ are disjoint divisors. Also, \eqref{cond ideals} implies that  
$$\sigma(P_{\alpha_i})=P_{\alpha_{i+1}} (1\le i \le n-1), \quad \sigma(P_{\alpha_n})=P_{\alpha_{1}}, \quad \sigma(P_\beta)=P_\beta, \quad \text{and} \quad \sigma(P_\infty) = P_\infty.$$
Thus, $\sigma(D)=D$ and $\sigma(G)=G$. Therefore, the AG-code $C_\rr(D,G)$ is cyclic, by Lemma~\ref{ej cic}.
The assertions on the parameters are straightforward from \eqref{dist} and \eqref{kG}.
\end{proof}

\begin{rem} \label{n cycle}
	If the first condition in \eqref{cond ideals} is replaced by
	$\sigma(x-\alpha_i) \in P_{\alpha_{\tau(i)}}$, $1\le i\le n$,
	for some $n$-cycle $\tau \in \mathbb{S}_n$ different from  {$(12\cdots n)$}, we get an AG-code which is equivalent to a cyclic one
	by Remark \ref{relabeling}. 
\end{rem}

Next, we give some examples of rational sigma-cyclic AG-codes using factorization of polynomials.
\begin{exam}[Frobenius] \label{frobenius}
Let $F=\ff_{q}(x)$ with $q=p^n$, $p$ prime and $n\ge 2$ (otherwise the construction would be trivial). For any primitive element $\alpha \in \ff_{p^n}$ take
$$p(x)=m_\alpha(x) = (x-\alpha)(x-\alpha^p) \cdots (x-\alpha^{p^{n-1}}) .$$
Let $\sigma \in  \aut(F)$ be the automorphism determined by
	\begin{equation*} \label{tau 1}
	\sigma(x)=x \qquad \text{and} \qquad \sigma(a)=a^p, \quad a\in \ff_{p^n}
	\end{equation*}
and extended by linearity to the whole $F$.
Then $\sigma(P_0)=P_0$ and $\sigma(P_\infty)=P_\infty$.
Also, since $\alpha^{p^n}=\alpha$ and $\sigma(x-\alpha^{p^i}) =	x-\alpha^{p^{i+1}}$
we have that $\sigma(P_{\alpha_i})=P_{\alpha_{i+1}}$ where $\alpha_i=\alpha^{p^{i-1}}$ for every $i$ mod $n$. 
In this way, $\cod = C_\rr(D,G)$, where $D=P_1+\cdots +P_n$ and $G=rP_0+sP_\infty$ with $r,s\in \z$ such that $r+s<n$, is a cyclic AG-code of length $n$ over $\ff_{p^n}$ by Lemma \ref{lemin ideals}, with $d>0$ and $k\ge r+s-1$. \hfill $\lozenge$
\end{exam}

Since $\sigma \in \aut(\ff_q(x))$ restricted to $\ff_q$ is an $\ff_p$-automorphism, and $\gal(\ff_q/\ff_p)$ is a cyclic group generated by the Frobenius automorphism, then for any $\alpha\in \ff_q$ we have $\sigma(\alpha) = \alpha^{p^\ell}$ for some 
$1\le \ell \le n$. This says that the cyclic code $\cod$ constructed in Example~\ref{frobenius} is essentially the only possible one with $\sigma(x)=x$. Since $p$ is prime, changing $\sigma(\alpha)=\alpha^p$ by $\sigma(\alpha)= \alpha^{p^\ell}$ for any 
$1 < \ell < n$, we get cyclic AG-codes all equivalent to $\cod$.

\begin{exam}[Roots of unity] \label{ex omega code}
Let $q$ be a prime power and let  $n\geq 2$ be a divisor  of $q-1$. Then $\ff_q$ contains a primitive $n$-th root of unity $\omega$ and hence $x^n-1$ splits into linear factors
	$$x^n-1 = (x-1)(x-\omega)(x-\omega^2)\cdots (x-\omega^{n-1})$$
in $\ff_q[x]$.
This factorization gives rise to $n$ rational places $P_1,P_{\omega},\ldots, P_{\omega^{n-1}}$ of $\ff_q(x)$. 
Let $\sigma$ be the automorphism in $\aut_{\ff_q}(\ff_q(x))$ induced by 
\begin{equation} \label{tau omega}
		\sigma(x) = \omega^{-1} x \qquad \text{and} \qquad \sigma(a)=a, \quad a\in \ff_q.
\end{equation}
	
We clearly have that $\sigma(x) \in P_0$. Also, note that
	$$\sigma(x-\omega^i)= (\omega^{-1} x-\omega^i)= \omega^{-1} (x-\omega^{i+1})$$
and hence $\sigma(x-\omega^i) \in P_{\omega^{i+1}}$ for every $1\le i \le n$.
Therefore, taking $\alpha_i = \omega^{i-1}$ for any $i=1,\ldots,n$, we have that
$\sigma(P_{\alpha_i}) = P_{\alpha_{i+1}}$ for $i=1, \ldots,n$, $\sigma(P_0)=P_0$ and $\sigma(P_\infty)=P_\infty$. 
Therefore, $\cod = C_\rr(D,G)$ with $D = P_{\alpha_1} + \cdots + P_{\alpha_n}$, $G=rP_0 + sP_\infty$ and $r,s \in \z$,
is a $\sigma$-cyclic AG-code over $\ff_q(x)$ of length $n$ over $\ff_q$, by Lemma \ref{lemin ideals}.
\hfill $\lozenge$	
\end{exam}

\begin{exam}[Artin-Schreier polynomial] \label{ej 3}
Consider the polynomial $f(x)=x^p-x-a $ in $\ff_q[x]$, with $p=\mathrm{char}(\ff_q)$. 
Let $\alpha \in \ff_q \smallsetminus \ff_p$ and take $a=\alpha^p-\alpha$ (hence $a\ne 0$). Thus $\alpha$ is a root of $x^p-x-a$ and it is easy to see that $\alpha+1$ is also a root of $f(x)$. We clearly have
	$$x^p-x-a = (x-\alpha)\big(x-(\alpha+1)\big) \cdots \big(x-(\alpha+p-1)\big).$$ 
Consider the $\ff_q$-automorphism determined by 
	$$\sigma(x)=x-1.$$  
Taking $\alpha_i = \alpha+i-1$ for $i=1,\ldots, p$, we see that $\sigma(P_{\alpha_i})=P_{\alpha_{i+1}}$ for $i=1,\ldots,p$ and $\sigma(P_\infty)=P_\infty$.
Then, by Lemma \ref{ej cic}, the code $\cod=C_\rr(D,G)$ where $D=P_{\alpha_1}+\cdots+P_{\alpha_p}$ and $G=sP_\infty$ with 
$s\in \na$ 
is cyclic of prime length $p$. \hfill $\lozenge$
\end{exam}

We have given easy examples of cyclic rational AG-codes. In Section 4 we will study the construction of sigma-cyclic AG-codes over $\ff_q(x)$ in a more systematic way.

\section{Examples of cyclic AG-codes through cyclic extensions}
In the previous section we gave examples of cyclic AG-codes over the rational function field. 
Perhaps the simplest way to obtain concrete examples of sigma-cyclic AG-codes over function fields of positive genus is by considering a cyclic extension $F'/F$ of function fields over $\ff_q$ and the subgroup $\mathrm{Gal}(F'/F)$ of $\aut_{\ff_q}(F')$. In fact, we will show in Section 5 that sigma-cyclic AG-codes are basically obtained from cyclic extensions.

We begin by showing that  cyclic extensions are suitable to construct sigma-cyclic AG-codes with the sigma-method. 

\begin{prop}\label{sigmacyclic} 
Let $F'/F$ be a cyclic extension  of degree $m$ of function fields over $\ff_q$. Let $P$ be a place of $F$ and let $P_1,\ldots,P_n$ be all the places of $F'$ lying over $P$. Then, $n$ divides $m$  and for any generator $\sigma$ of $\gal(F'/F)$ we have that the orbit of $P_1$ is  
$$[P_1]_\sigma=\{P_1,\ldots,P_n\}.$$
Furthermore, let $Q \neq P$ be a place of $F$ and let $G=Q_1+\cdots+Q_k$ be the divisor of $F'$ formed with all the places of $F'$ lying over $Q$.  Then $\sigma(G)=G$ and $\sop (G) \cap \sop(D) = \varnothing$ where  $D=P_1+\cdots+P_n$. In particular, if each $P_i$ is rational then $C_\rr(D,G)$ is a sigma-cyclic AG-code. 
\end{prop}

\begin{proof}
Suppose $F'/F$ is a cyclic extension of degree $m$ and let $\sigma$ be a generator of the Galois group $\mathcal{G}=\gal(F'/F)$. 
First, notice that $n \mid m$ because $ne \!f=m$ where $e=e(P_i \,|\, P)$ and $f=f(P_i \,|\, P)$ are the ramification index and the inertia degree, respectively, for $i=1,\ldots,n$. 
	
We now show that the orbit $[P_1]_\sigma$ consists of the places $P_1,\ldots,P_n$. Consider the decomposition group 
	$$  D(P_1 \,|\, P) = \{\sigma \in \mathcal{G} : \sigma(P_1)=P_1\}$$
of $P_1$ over $P$.	If $\sigma^i(P_1)=\sigma^j(P_1)$ for some $1\leq i<j\leq n$ then we have 
$\sigma^k\in D(P_1 \,|\, P)$ for $k=j-i>0$.
Since $D(P_1 \,|\, P)$ is a group of order $ef$  we have that $\sigma^{ke\!f}=id$, the identity element of $\gal(F'/F)$. But $k<n$ and so $ke\!f<ne\!f=m$ contradicting that $m$ is the order of $\sigma$. Therefore the set $\{\sigma^i(P_1)\}_{i=1}^n$ consists of $n$ different places of $F'$ lying over $P$ and thus we must have
	$$ \{\sigma(P_1), \sigma^2(P_1), \ldots, \sigma^n(P_1)\} = \{P_1,P_2,\ldots,P_n\} .$$
This means that $\sigma^j(P_1)=P_1$ for some $1\leq j\leq n$. But the above argument implies that we can not have $1\leq j\leq n-1$. Therefore $\sigma^n(P_1)=P_1$  and we are done with the first part.
	
It is clear by construction that $\sigma(G)=G$ and $\sop (G) \cap \sop(D) = \varnothing$. Therefore if each $P_i$ is rational
 then $\cod_\rr(D,G)$ is a sigma-cyclic AG-code.
\end{proof}

We now give some explicit constructions of sigma-cyclic AG-codes obtained from cyclic extensions, namely Kummer,  Artin-Schreier and Hermitian extensions.

\begin{exam}[\textit{Kummer extensions}]
Consider the rational function field $F=\ff_q(x)$ and let $F'=F(y)$ be the Kummer extension of $F$ given by
	$$y^n = (x-\alpha)(x-\alpha^{-1})$$
where $n\mid q-1$, $\alpha\in \ff_q^*$ and $\alpha\ne \alpha^{-1}$.
By Proposition 6.3.1 in \cite{Sti} we have that $F'/F$ is cyclic of degree $n$ and $\ff_q$ is the full constant field of $F'$. Also, the places $P_{\alpha}$ and $P_{\alpha^{-1}}$, the zeroes of $x-\alpha$ and $x-\alpha^{-1}$ respectively, are totally ramified in $F'/F$.
	
Note that $P_0$ splits completely in $F$. In fact, let 
	$$\varphi(T) = T^n - (x-\alpha)(x-\alpha^{-1}) \in \ff_q(x)[T]$$ 
and let $\bar{\varphi}(T)$ be its reduction mod $P_0$, the zero of $x$ in $F$. Since $x(P_0)=0$ and $n\mid q-1$ then
	$$\bar{\varphi}(T)= T^n -1 = \prod_{i=1}^n(T-a_i) \in \ff_q[T] \,. $$
Therefore, by Kummer's Theorem, $P_0$ splits completely in $F$.
	
Now, let $D=P_1+\cdots + P_n$, where $P_1,\ldots,P_n$ are all the places of $F'$ lying over $P_0$, and 
let $G= r Q_\alpha$ where $r$ is a positive integer and $Q_\alpha$ is the only place of $F'$ lying over $P_\alpha$. 
By Proposition \ref{sigmacyclic} we have that the AG-code $\cod=C_\rr(D,G)$ is sigma-cyclic. 
	
We have estimates for the parameters $[n,k,d]$ of $\cod$. The genus $g$ of $F'$ is $g = [\tfrac{n-1}{2}]$. 
Thus, if $0<r<n$, then $d$ and $k$ satisfy the following inequalities:
	$$ d\geq n-r \qquad \text{and} \qquad k\geq r+1-[\tfrac{n-1}{2}], $$
by \eqref{dist} and \eqref{kG}. \hfill $\lozenge$
\end{exam}

\begin{exam}[\textit{Artin-Schreier extensions}]
Let $p$ be an odd prime number. We consider  the Artin-Schreier extension $F'/\ff_p(x)$ where $F'=\ff_p(x,y)$ and 
$$y^p-y=x^2.$$ 
From Proposition 3.7.8 in \cite{Sti90} we have that $F'$ is a cyclic extension of $\ff_p(x)$ of degree $p$, the place $P_\infty$ is totally ramified in $F'$ and any other place of $\ff_p(x)$ is unramified in $F'$. Since \[T^p-T=T^p-T-x\mod P_0,\] we have from Kummer's theorem that $P_0$ splits completely in $F'$ into $p$ rational places $P_1,\ldots, P_p$. 
	
If we take $D=P_1+\cdots+P_p$ and $G=rQ$, where $Q$ is the only place of $F'$ lying over $P_\infty$ and $1\leq r\leq p-1$, then, by Proposition \ref{sigmacyclic}, the AG-code $C_\rr(D,G)$ is sigma-cyclic. Since the genus of $F'$ is $g=\tfrac 12 (p-1)$, from \eqref{dist} and \eqref{kG} we have that 
	\[ d\geq p-r  \qquad \text{and} \qquad k \geq \tfrac 12 (2r+1-p).\]
We see that, in fact, we must have $\tfrac 12 (p+1) \leq r\leq p-1$. \hfill $\lozenge$
\end{exam}

\begin{exam}[\textit{Cyclic codes from Hermitian function fields}]
Let $H=\ff_{q^2}(x,y)$ be the Hermitian function field, extension of $\ff_{q^2}(x)$, defined by 
\begin{equation} \label{Hermitian}
y^{q+1}=x^{q+1}-1.
\end{equation}
From Examples 6.3.5 and 6.3.6 of \cite{Sti} we have that $H$ is a cyclic extension of degree $q+1$ of $\ff_{q^2}(x)$ which is also a maximal function field of genus $g=\frac 12 q(q-1)$. If $\alpha\in \ff_{q^2}$ is such that $\alpha^{q+1}=1$ then the rational place $P_\alpha$ of $\ff_{q^2}(x)$ is totally ramified in $H$. On the other hand  if $\alpha\in \ff_{q^2}$ is such that $\alpha^{q+1}\neq 1$, then the rational place $P_\alpha$ of $\ff_{q^2}(x)$ splits completely in $H$. Also the pole $P_\infty$ of $x$ in $\ff_{q^2}(x)$ splits completely in $H$.

Now let $\alpha\in \ff_{q^2}$ such that $\alpha^{q+1}=1$ and let $Q_\alpha$ be the only rational place of $H$ lying over $P_\alpha$. If we consider the divisors $D=P_1+\cdots+P_{q+1}$, where each place $P_i$ lies over $P_\infty$, and $G=rQ_\alpha$ with $1\le r\le q-1$ we have, by Proposition \ref{sigmacyclic}, that
$C_\rr(D,G)$ is a sigma-cyclic code of length $q+1$ and minimum distance $d\geq q-r$ defined over the Hermitian function field $H$.  \hfill $\lozenge$
\end{exam}

We finish this section with an interesting consequence of Lema \ref{ej cic} and Proposition \ref{sigmacyclic}.

\begin{coro}\label{corosigmacyclic}
	Let $F'/F$ be a Galois extension of degree $n$ of function fields over $\ff_q$. Let $P_1,\ldots,P_n$ be $n$ different places of $F'$. Suppose that \eqref{conditions} holds with the places $P_1,\ldots,P_n$ for some  $\sigma\in \gal(F'/F)$. Then the extension $F'/F$ is cyclic, $\sigma$ generates $\gal(F'/F)$ and there is a place $P\in\pl(F)$ such that $P_i\cap F=P$ for $i=1,\ldots, n$, i.e.\@ $P$ splits completely in $F'$. Reciprocally if
	$F'/F$ is cyclic and some place $P\in\pl(F)$ splits completely in $F'$ into $P_1,\ldots,P_n$, then \eqref{conditions} holds with these places for any generator $\sigma$ of $\gal(F'/F)$.
\end{coro}

\begin{proof}
Let $\sigma\in \gal(F'/F)$ such that \eqref{conditions} holds for the $n$ rational places $P_1,\ldots,P_n$. From Lemma \ref{ej cic} we have that $n$ divides the order of $\sigma$. Since the $\gal(F'/F)$ is a group of order $n$ we have that $\sigma$ generates $\gal(F'/F)$ and thus $F'/F$ is a cyclic extension. If $P$ is a place of $F$ lying above $P_1$ then, since $\sigma$ restricted to $F$ is the identity, every place in the orbit $[P_1]_\sigma$ lies above $P$. But we know that $[P_1]_\sigma=\{P_1,\ldots,P_n\}$, then we have that every place $P_i$ lies above $P$ and this means that $P$ splits completely in $F'$.
 
The reciprocal implication follows immediately from Proposition~\ref{sigmacyclic}. 
\end{proof}

\section{ One point sigma-cyclic rational codes} \label{sigmarational}
In this section we study the case of one point sigma-cyclic AG-codes over a rational function field $\ff_q(x)$. This kind of cyclic codes will be called sigma-cyclic rational codes (over $\ff_q(x)$).  Besides studying the problem of constructing sigma-cyclic rational codes, we will also study the problem of the monomial equivalence (see Definition \ref{defi}) of sigma-cyclic rational codes of the form $C_\rr(D,rP)$ where $P$ is a rational place of $\ff_q(x)$. We will show that for a fixed length and dimension any sigma-cyclic rational code $C(D,rP)$, with $P$ a rational place of $\ff_q(x)$, is equivalent to one of the form $C_\rr(D',rP_\infty)$ where $P_\infty$ is the only pole of $x$ in $\ff_q(x)$.

Let us recall that the group of $\ff_q$-automorphism of $\ff_q(x)$ can be identified with the projective linear group of $2 \times 2$ matrices over $\ff_q$, that is  
$$\aut_{\ff_q}(\ff_q(x)) \simeq \pgl_2(\ff_q),$$
and if $\sigma\in \aut_{\ff_q}(\ff_q(x))$  then
$$ \sigma(x)=\frac{ax+b}{cx+d},$$
for some $a, b, c, d\in\ff_q$ such that
$\left( \begin{smallmatrix} a & b\\ c & d \end{smallmatrix} \right) \in \pgl_2(\ff_q)$.

\sk 
Let $A=(\begin{smallmatrix} a & b\\ c & d \end{smallmatrix}) \in \pgl_2(\ff_q)$
with $A \ne I$, where $I$ is the identity matrix of $\pgl_2(\ff_q)$. Its inverse in $\pgl_2(\ff_q)$ is $A^{-1} = 
\left(\begin{smallmatrix} d & -b\\ -c & a  \end{smallmatrix}\right)$.   
 
The action of $A^{-1}$ on $\mathbb{P}^1(\ff_q)=\ff_q \cup \{\infty\}$
is as a fractional transformation, that is, if $\alpha\in \ff_q$ we have 
\[ A^{-1} \cdot \alpha = 
\left\{ \begin{array}{cl} \frac{d\alpha-b}{-c\alpha +a} & \qquad \text{if $a\neq c\alpha$,} \\[2mm] 
\infty & \qquad \text{if $a=c\alpha$,} \end{array} \right. \]
and
\[ A^{-1} \cdot \infty = 
\left\{ \begin{array}{cl} -dc^{-1} & \qquad \text{if $c\neq 0$,} \sk \\
\infty & \qquad \text{if $c=0$.} \end{array} \right. \]

In Section 2 we gave examples of sigma-cyclic rational codes. In these examples, the involved matrices $A$ are dilations and translations, namely $(\begin{smallmatrix}
	\omega^{-1} & 0 \\ 0 & 1 \end{smallmatrix})$ in Example 2.9, where $\omega$ is a $n$-th root of unity, and $(\begin{smallmatrix} 1 & -1 \\ 0 & 1 \end{smallmatrix})$ in Example 2.10. 

We now begin with the construction of general sigma-cyclic rational codes, that is with fractional transformations. 
If $A\in \pgl_2(\ff_q)$ is of order $n$, then the orbits of the action of $A^{-1}$ on  non fixed points of  $A$ will give rise to $n$ rational places of $\ff_q(x)$ which are cyclically permuted by the $\ff_q$-automorphism of $\ff_q(x)$ associated to $A$.  More precisely:
 
\begin{prop} \label{automorfismos} 
Let $A = \left(\begin{smallmatrix} a & b\\ c & d \end{smallmatrix}\right) \in \pgl_2(\ff_q)$ a matrix of order $n>1$. Let $\alpha\in \mathbb{P}^1(\ff_q)=\ff_{q}\cup\{\infty\}$ such that $A\cdot\alpha\neq\alpha$. Let $[\alpha]_A=\{\alpha_1,\ldots,\alpha_n\}$ 
be the orbit of $\alpha$ under the action of $A^{-1}$, that is  $\alpha_1=\alpha$ and
$\alpha_{i+1}=A^{-i} \cdot \alpha_1$ for $i=1,\ldots,n-1$. 
Let $\sigma$ be the $\ff_q$-automorphism of $\ff_q(x)$  corresponding to $A$. Let $P_i$ be either the zero of $x-\alpha_i$ in $\ff_q(x)$ if $\alpha_i\in\ff_q$ or the pole of $x$ in $\ff_q(x)$ if $\alpha_i=\infty$ so that $P_i=P_{\infty}$ in this case.  Then we have:  
\begin{enumerate}[$(a)$]
\item $\alpha_1,\ldots,\alpha_n$ are all distinct elements of $\mathbb{P}^1(\ff_q)$. \sk 

\item $P_i\neq P_j$ for $i\neq j$. \sk 

\item $\sigma(P_i)=P_{i+1}$ for $i=1,\ldots,n-1$ and $\sigma(P_n)=P_1$.
\end{enumerate}
Furthermore $\sigma(P_{\infty})=P_{\infty}$ if and only if $c=0$.
\end{prop}

\begin{proof}
From Lemma 2.3 in \cite{LMNX02} we have that $\{\alpha_i\}_{i=1}^n$ is a sequence of $n$ different elements of $\mathbb{P}^1(\ff_q)$  because $A^{-1}$ is of order $n>1$ and $\alpha_1$ is not a fixed point of $A^{-1}$. This implies that $P_i\neq P_j$ for $i\neq j$, thus proving $(a)$ and $(b)$. 

We now show $(c)$. For $\alpha_i\in\ff_q$ we have   
\[ \sigma(x-\alpha_i) = \sigma(x)-\sigma(\alpha_i) = \frac{ax+b}{cx+d}-\alpha_i = \frac{(a-c\alpha_i)x-(d\alpha_i-b)}{cx+d}. \]
If $a\neq c\alpha_i$, then $A^{-1}\cdot \alpha_i=\alpha_{i+1}\in\ff_q$ and we have that
\[ \sigma(x-\alpha_i) = (a-c\alpha_i) \,\frac{x-A^{-1} \cdot \alpha_i}{cx+d} = \frac{a-c\alpha_i}{cx+d} \, (x-\alpha_{i+1}). \]
This implies that $\sigma(x-\alpha_i)\in P_{i+1}$ which  means that $\sigma(P_i)=P_{i+1}$. 
Now, if $a=c\alpha_i$, then $\alpha_{i+1}=A^{-1}\cdot \alpha_i=\infty$ and $c\neq 0$, and the above computation shows that
\[\sigma(x-\alpha_i)=-\frac{d\alpha_i-b}{cx+d} \in P_\infty. \]
Then $\sigma(x-\alpha_i)\in P_{i+1}$ where $P_{i+1}=P_{\infty}$ and hence $\sigma(P_i)=P_{i+1}$. 

\sk
On the other hand, if $\alpha_i=\infty$ then $P_i=P_{\infty}$ and $c\neq 0$ because $\{\alpha_i\}_{i=1}^m\subset \ff_q$ when $c=0$. Therefore 
\[\sigma(x^{-1})=\frac{cx+d}{ax+b}=c\,\frac{x+dc^{-1}}{ax+b}=c\,\frac{x-A^{-1}\cdot \infty}{ax+b}=c\,\frac{x-\alpha_{i+1}}{ax+b}\] 
so that $\sigma(x^{-1})\in P_{i+1}$ and thus $\sigma(P_i)=P_{i+1}$. Since $A^{-1}\cdot \alpha_m=A^{-m}\cdot\alpha_1=\alpha_1$ the above computations for $i=m$ show that $\sigma(P_m)=P_1$. 

Finally, notice that if $c=0$ then $\sigma(P_{\infty})=P_{\infty}$ because
\[\sigma(x^{-1})=\frac{d}{ax+b}\in P_{\infty}. \] 
Reciprocally, if $\sigma(P_\infty)=P_{\infty}$, then 
\[\sigma(x^{-1})=\frac{cx+d}{ax+b} \in P_{\infty},\]
since $x^{-1}\in P_\infty$, which is impossible unless $c=0$. 
This completes the proof.
\end{proof}

Under the conditions of the previous proposition, if $D=P_1+\cdots +P_n$ then $\sigma(D)=D$. Therefore, we will have a sigma-cyclic rational code $C_{\rr}(D,G)$ over $\ff_q$   once we find a suitable divisor $G$ invariant under $\sigma$.

\begin{exam}
Let $F=\ff_4(x)$ with $\ff_4$ generated by $\beta$ such that $\beta^2+\beta+1=0$. Take
$A=\left(\begin{smallmatrix} 1  & 1\\ \beta & 0  \end{smallmatrix}\right) \in \pgl_2(\ff_4)$. 
The $\ff_4$-automorphism of $\ff_4(x)$ associated to $A$ is 
$$\sigma(x)=\frac{x+1}{\beta x}.$$ 

The matrix $A$ has order 5 in $\mathrm{PGL}(2,\ff_4)$ with inverse 
$A^{-1} = \left( \begin{smallmatrix}  0 & 1\\ \beta & 1 \end{smallmatrix} \right)$. 
It is easy to check that if we take $\alpha_1 = 1$ and we define $\alpha_{i+1}=A^{-1}\cdot \alpha_i$ for $i\ge 1$, then we get 
$$ \alpha_2 = \beta, \quad \alpha_3 = \beta+1, \quad \alpha_4 = \infty, \quad \alpha_5=0 $$
(and of course $\alpha_6 = \alpha_1$).

For each $1\le i \le 5$, let $P_i$ be the rational place defined by the zero of $x-\alpha_i$ in $\ff_q(x)$ if $\alpha_i\ne \infty$ or else 
$P_i=P_\infty$, the pole of $x$ in $\ff_q(x)$. By Proposition \ref{automorfismos} we have that $P_1,\ldots,P_5$ are all different and $\sigma(P_i) = P_{i+1}$ for $i=1,\ldots,4$ and $\sigma(P_5)=P_1$.

If $D=P_1+P_2+P_3+P_4+P_5$, then $C_{\rr}(D,G)$ will be a cyclic AG-code provided we can find a divisor $G$ disjoint with $D$, 
hence with support consisting of non-rational places of $\ff_4(x)$, such that $G$ is invariant under $\sigma$. 
A convenient choice is, for instance, to consider the place $Q$ of degree $2$ defined by the monic irreducible polynomial 
$x^2+\beta^2x+\beta^2$, that is $Q=P_{x^2+\beta^2x+\beta^2}$. A direct computation shows that 
  \[\sigma(x^2+\beta^2x+\beta^2) = \frac{\beta^2}{x^2} \, (x^2+\beta^2x+\beta^2) \in Q,\]
  so that $\sigma(Q)=Q$ and we can take $G=rQ$ where $r\in\na$. 

In this way we have constructed a sigma-cyclic rational code $C_{\rr}(D,G)$ over $\ff_4$ of length $5$ by using the whole set of rational places of $\ff_4(x)$. \hfill $\lozenge$
\end{exam}
 
We now use the previous proposition to show that, for every prime power $q$, there are sigma-cyclic rational codes of the form $C_{\rr}(D,rP_\infty)$ over $\ff_{q}$ of length $q-1$ where $0 \leq r \leq n-2$.

\begin{exam}
Let $a$ be a primitive element of the multiplicative group $\ff_{q}^{*}$ and let $A=\left(\begin{smallmatrix} 1 & 0\\ 0 & a \end{smallmatrix}\right) \in \pgl_{2}(\ff_{q})$. The inverse of $A$ in $\pgl_{2}(\ff_{q})$ can be written as $A^{-1}=\left(\begin{smallmatrix} a & 0\\ 0 & 1 \end{smallmatrix}\right)$. It is easy to check that the order of $A$ in $\pgl_{2}(\ff_{q})$ is $q-1$ (alternatively see Lemma \ref{orden}). On the other hand $A^{-1}\cdot 1=a\neq 1$ because $a$ generates $\ff_{q}^{*}$. Therefore, the orbit of $1$ has $q-1$ elements, that is
$$[1]_{A}=\{1,a,a^{2},\ldots,a^{q-2}\}=\ff_{q}^{*}.$$

Let $n=q-1$ and let us write $P_{i}=P_{x-a_{i-1}}$ for $1 \leq i \leq n$. Let $\sigma$ be the $\ff_q$-automorphism of $\ff_{q}(x)$ associated to $A$. Using Proposition \ref{automorfismos} we see that $\sigma(P_{i})=P_{i+1 \mod n}$ and $\sigma(P_{\infty})=P_{\infty}$. In this way we have the  sigma-cyclic rational code $C_{\rr}(D,rP_\infty)$ of length $n=q-1$ where  $D=P_{1}+P_{2}+\cdots + P_{n}$ and $0 \leq r \leq n-2$.  
This generalizes Example 2.9. 
	
Similarly, we can define a sigma-cyclic rational code of length $n=q-1$ by using a matrix 
of the form $A=\left(\begin{smallmatrix} a & b\\ 0 & d \end{smallmatrix}\right) \in \pgl_{2}(\ff_{q})$. This generalizes Example 2.10. \hfill $\lozenge$
\end{exam}

\subsubsection*{Monomial equivalence of sigma-cyclic rational codes.} We have mentioned in the introduction the problem of the  equivalence of  cyclic  AG-codes over a function field $F$. We consider now this problem in the special case of cyclic AG-codes over $F=\ff_q(x)$ constructed with the sigma-method.  We know from Proposition 2.2.14 in \cite{Sti} that if two divisors $G_1$ and $G_2$ of a function field $F$ over $\ff_q$ are equivalent, that is, there exists a principal divisor $(z)$ of $F$ such that $G_2=G_1+(z)$, then the AG-codes $C_\rr(D,G_1)$ and $C_\rr(D,G_2)$ are equivalent provided $\sup G_i\cap \sup D=\varnothing$ for $i=1,2$.  This was used in \cite{LN99} to classify rational AG-codes over $\ff_q$.

\begin{rem}
	The equivalence of codes considered in \cite{Sti} is more restrictive than the one we use here: two linear codes are considered equivalent in \cite{Sti} if condition $(b)$ of Definition~\ref{defi} holds.
\end{rem} 

In our case we will need a less general but more precise result. Let $\beta\in \ff_q^*$ and let $\tau_\beta$ be the $\ff_q$-automorphism of $\ff_q(x)$ associated to the matrix
$(\begin{smallmatrix} 1 & \beta \\ 0 & 1 \end{smallmatrix})$ so that 
$$\tau_\beta(x)=x+\beta.$$

\begin{lem} \label{taurr}
Let $\beta\in \ff_q^*$ and let us consider the rational places $P_\beta$ and $P_0$ of $\ff_q(x)$. Then, for any integer $r\geq 1$ we have 	
\[ \tau_\beta(\rr(rP_\beta))=\rr(rP_0). \]
\end{lem}

\begin{proof}
We clearly have that  $\tau_\beta(P_\beta)=P_0$.	Recall that for any place $P$ of $\ff_q(x)$
	\[\rr(rP)=\{z\in \ff_q(x): (z)+rP\geq 0\}\cup \{0\}.\]
	Now let $z\in \rr(rP_\beta)$. Then
	$(z)=(z)_0-(z)_{\infty}=\sum n_PP-sP_\beta$,
	where $r\geq s\geq 1$. We have that
	\[\nu_{P_0}(\tau_\beta(z))=\nu_{\tau_\beta^{-1}(P_0)}(z)=\nu_{P_\beta}(z)=-s,\]
	and if $\nu_Q(\tau_\beta(z))<0$ then $Q=P_0$ because
	\[0>\nu_Q(\tau_\beta(z))=\nu_{\tau_\beta^{-1}(Q)}(z),\]
	so that $\tau_\beta^{-1}(Q)=P_\beta$. Thus, we have $(\tau_\beta(z))_\infty=-sP_0$ with $r\geq s$ and this implies that $\tau_\beta(z)\in \rr(rP_0)$.
	
	Reciprocally, if $y\in \rr(rP_0)$ then 
	$(y)=(y)_0-(y)_{\infty}=\sum n_PP-sP_0$,
	where $r\geq s\geq 1$. We have that
	\[\nu_{P_\beta}(\tau_\beta^{-1}(y))=\nu_{\tau_\beta(P_\beta)}(y)=\nu_{P_0}(y)=-s, \]
	and if $\nu_Q(\tau_\beta^{-1}(y))<0$ then $Q=P_\beta$ because
	\[0>\nu_Q(\tau_\beta^{-1}(y))=\nu_{\tau_\beta(Q)}(y),\]
	so that $\tau_\beta(Q)=P_0$. Therefore we have that $(\tau_\beta^{-1}(y))_{\infty}=-sP_\beta$ with $r\geq s$ and this implies that $\tau_\beta^{-1}(y)=z\in \rr(rP_\beta)$.	
\end{proof}

As a direct consequence of the above result we have

\begin{prop}\label{equality}
Let $P_1,\ldots,P_n$ be $n$  rational places of $\ff_q$ and let $\beta\in \ff_q^*$. Then we have	
\[C_\rr(D,rP_\beta)=C_\rr(\tau_\beta(D),rP_0),\]
where $D=P_1+\cdots+P_n$ and $\tau_\beta(D)=\tau_\beta(P_1)+\cdots+\tau_\beta(P_n)$.
\end{prop}
\begin{proof} From Lemma \ref{taurr} we have each element $y\in \rr(rP_0)$ is of the form $\tau_\beta(z)$ for some $z\in \rr(rP_\beta)$. Thus $C_\rr(\tau_\beta(D),rP_0)\subset C_\rr(D,rP_\beta)$ because
	\begin{align*}
		(y(\tau_\beta(P_1)),\ldots,y(\tau_\beta(P_n))) &= (\tau_\beta(z)(\tau_\beta(P_1)),\ldots,\tau_\beta(z)(\tau_\beta(P_n)))\\
		&=(z(\tau_\beta^{-1}(\tau_\beta(P_1))),\ldots, z(\tau_\beta^{-1}(\tau_\beta(P_n)))\\
		&=(z(P_1),\ldots,z(P_n)).
	\end{align*}
Reciprocally, from Lemma \ref{taurr} each element $z\in \rr(rP_\beta)$ is of the form $\tau_\beta^{-1}(y)$ for some $y\in \rr(rP_0)$. Then $C_\rr(D,rP_\beta)\subset C_\rr(\tau_\beta(D),rP_0)$ because
\begin{align*}
	(z(P_1),\ldots,z(P_n)) = (\tau_\beta^{-1}(y)(P_1),\ldots,\tau_\beta^{-1}(y)(P_n)) =(y(\tau_\beta(P_1)),\ldots,y(\tau_\beta(P_n))),
\end{align*}
and we are done.
\end{proof}

\begin{rem}\label{equalitysigma}
	If the AG-code $C_\rr(D,rP_\beta)$ in Proposition \ref{equality} is sigma-cyclic, so that $\sigma(P_i)=P_{i+1}$ for $1\leq i\leq n-1$ and $\sigma(P_n)=P_1$ for some $\sigma\in\aut_{\ff_q}(\ff_q(x))$, then it is easy to check that  $C_\rr(\tau_\beta(D),rP_0)$ is also sigma-cyclic with respect to the $\ff_q$-automorphism $\tau_\beta\sigma\tau_\beta^{-1}$.
\end{rem}   

\goodbreak 
The sigma-cyclic rational codes we are concerned with are of the following form:
\begin{enumerate}
	\item[\textbf{(A)}] \: $C_\rr(D,rP_\beta)$ with $\beta\in \ff_q^*$, \sk 
	
 	\item[\textbf{(B)}] \: $C_\rr(D,rP_0)$ and \sk 
 	
	\item[\textbf{(C)}] \: $C_\rr(D,rP_\infty)$.
\end{enumerate}
From Proposition \ref{equality} and Remark \ref{equalitysigma} we see that the problem of the equivalence of sigma-cyclic rational codes is reduced to the cases \textbf{(B)} and \textbf{(C)}.
 
Let $A\in \pgl_{2}(\ff_{q})$ be a matrix of order $n$ and let $\alpha, \beta \in \mathbb{P}^1(\ff_{q})$ be such that $A^{-1}\cdot \alpha \ne \alpha$ and $A^{-1}\cdot \beta = \beta$.  Let $1\leq r \leq n-2$ be an integer and let us consider the orbit of $\alpha$ under the action of $A^{-1}$, i.e.
\begin{equation}\label{orbit}
[\alpha]_A=\{\alpha_1,\alpha_2,\ldots,\alpha_n\},
\end{equation}
where $\alpha_1=\alpha$ and $\alpha_{i+1}=A^{-1}\cdot\alpha_i$. According to Proposition \ref{automorfismos}, associated to the orbit $[\alpha]_A$ we have the rational places $P_{\alpha_1},\ldots,P_{\alpha_n}$ of $\ff_q(x)$ and they are all distinct.  With these places and the rational place $P_\beta$ of $\ff_q(x)$ we  have the divisors
\begin{equation} \label{divD}
D=P_{\alpha_1}+\cdots+P_{\alpha_n} \qquad \text{and} \qquad G=rP_\beta,
\end{equation}
whose supports are disjoint. In this way, by  considering the $\ff_q$-automorphism $\sigma$ of $\ff_q(x)$  associated to $A$, we have the sigma-cyclic rational code  $C_\rr(D,G)$. For the problem of the equivalence of the sigma-cyclic rational codes, it will be convenient to denote the code $C_\rr(D,G)$ defined above as
\[\cod(A,\alpha, \beta, r),\]
to emphasize its dependence  on $A$, $\alpha$, $\beta$ and $r$.

We begin with the case \textbf{(B)}. The $\ff_q$-automorphisms of $\ff_q(x)$ fixing the place $P_0$ are represented by matrices of $\pgl_2(\ff_q)$ fixing the point $0\in \mathbb{P}^1(\ff_q)$. Any matrix of $\pgl_2(\ff_q)$ with this property can always be written as
\begin{equation*} \label{fixing0}
	\begin{pmatrix}
		1 & 0\\
		c & d
	\end{pmatrix}.
\end{equation*}
We prove now that a sigma-cyclic rational code associated to a matrix of the form \eqref{fixing0} coincides with another one associated to a matrix fixing the point $\infty$. For this purpose we will need the following  $\ff_q$-basis of $\rr(rP_{\beta})$
\begin{equation} \label{baseB} 
B_{\beta} = \Big\{ 1, \frac{1}{x-\beta}, \frac{1}{(x-\beta)^{2}}, \ldots, \frac{1}{(x-\beta)^{r}} \Big\}
\end{equation}
where $\beta \in \ff_q$.  
\begin{prop} \label{equalityzero} Let $(\begin{smallmatrix}
		1 & 0 \\ c & d \end{smallmatrix}) \in \pgl_2(\ff_q)$ be of order $n$. Then
	$(\begin{smallmatrix} d & c \\ 0 & 1 \end{smallmatrix}) \in \pgl_2(\ff_q)$
	is of order $n$, fixes the point $\infty$ and
	\[\cod\left(\left(\begin{smallmatrix} 1 & 0 \\ c & d \end{smallmatrix}\right),\alpha , 0 , r\right) = \cod\left(\left(\begin{smallmatrix} d & c \\ 0 & 1 \end{smallmatrix}\right),\alpha^{-1} , \infty , r\right).\]
	In other words, for a given  sigma-cyclic rational code over $\ff_q$ of the form $C_\rr(D,rP_0)$ there exists a divisor $D'$ of $\ff_q(x)$ such that
	\[ C_\rr(D,rP_0)=C_\rr(D',rP_\infty).\] 
\end{prop}
\begin{proof}
	The first two assertions are immediate. We write
	$A=(\begin{smallmatrix}
		1 & 0\\
		c & d
	\end{smallmatrix})$ 
and $B=(\begin{smallmatrix}
		d & c\\
		0 & 1
	\end{smallmatrix})$.
	Let us consider the orbits $[\alpha]_A=\{\alpha_1,\ldots,\alpha_n\}$ and $[\alpha^{-1}]_B=\{\theta_1,\ldots,\theta_n\}$ where $\theta_1=\alpha^{-1}$. Let
	$$
	M_{1}=
	\begin{pmatrix}
		1 & 1 & \ldots &  1 \\
		\alpha_1^{-1}  & \alpha_2^{-1} & \ldots & \alpha_n^{-1} \\
		\alpha_1^{-2}  & \alpha_2^{-2} & \ldots & \alpha_n^{-2} \\
		\vdots            &      \vdots     & &   \vdots\\
		\alpha_1^{-r} & \alpha_2^{-r}  & \ldots & \alpha_n^{-r} \\
	\end{pmatrix}
	\qquad \text{ and } \qquad
	M_{2}=
	\begin{pmatrix}
		1            &1           &\ldots& 1     \\
		\theta_{1}   &\theta_{2}  &\ldots&\theta_{n} \\
		\theta_{1}^{2}&\theta_{2}^{2}&\ldots&\theta_{n}^{2}\\
		\vdots            &      \vdots     & &   \vdots\\
		\theta_{1}^{r}&\theta_{2}^{r}&\ldots&\theta_{n}^{r}\\
	\end{pmatrix}
	$$
	We have that $M_2$ is a generator matrix for the code $\cod(B,\alpha^{-1},\infty,r)$ and, by using the base $B_\beta$ in \eqref{baseB} with $\beta=0$,  we also have that $M_1$ is a generator matrix for the code $\cod\left(A,\alpha , 0 , r\right)$.

	The proposition is proved if we show that $\theta_i=\alpha_i^{-1}$ for $i=1,\ldots, n$ because, in this case, $M_1=M_2$. We proceed by induction. By definition we have that $\theta_1=\alpha^{-1}=\alpha_1^{-1}$.   Suppose now that  $\theta_{i}=\alpha_{i}^{-1}$. Then
	\begin{align*}
		\theta_{i+1}=B^{-1}\theta_{i}=\begin{pmatrix}
			1 & -c\\
			0 & d
		\end{pmatrix}\cdot \theta_i &=\frac{\theta_{i} - c}{d} =\frac{\alpha_i^{-1}-c}{d} =\frac{1-c\alpha_i}{d\alpha_i}.
	\end{align*}
	On the other hand
	\[ \alpha_{i+1}=A^{-1}\cdot \alpha_i=\begin{pmatrix}
		d & 0\\
		-c & 1
	\end{pmatrix}\cdot\alpha_i=\frac{d\alpha_i}{-c\alpha_i+1}.\] 
	Thus $\theta_{i+1}=\alpha_{i+1}^{-1}$, as desired.
\end{proof}

 We now consider case \textbf{(C)}. The $\ff_q$-automorphisms of $\ff_q(x)$ fixing the place $P_\infty$ are represented by matrices of $\pgl_2(\ff_q)$ fixing the point $\infty\in \mathbb{P}^1(\ff_q)$. Any matrix of $\pgl_2(\ff_q)$ with this property is of the form   
 \[\begin{pmatrix}
a   & b   \\
0   & d  \end{pmatrix}.\] 

\begin{rem}
	The matrix  $(\begin{smallmatrix} a&b \\ 0&d \end{smallmatrix})\in\pgl_2(\ff_q)$ can be written in $\pgl_2(\ff_q)$ as a matrix of the form $ (\begin{smallmatrix} 1&-b' \\ 0&a' \end{smallmatrix})$ where  $b'=-ba^{-1}$ and $a'=da^{-1}$. By writing the matrices of $\pgl_2(\ff_q)$ fixing the point $\infty\in \mathbb{P}^1(\ff_{q})$ in the form $A=(\begin{smallmatrix}
	1   & -b   \\
	0   & a  \end{smallmatrix})$ we get a cleaner expression for the powers of $A^{-1}$ and for certain generator matrices of the code $\cod(A,\alpha, \infty, r)$. 
\end{rem}

\begin{lem}\label{orden}
Let $A= (\begin{smallmatrix}
1   & -b   \\
0   & a  \end{smallmatrix})\in \pgl_2(\ff_q)$, $A\ne I$. The order of $A^{-1}$ in $\pgl_2(\ff_q) $ is 
\[ |A^{-1}|= \left\{
\begin{tabular}{cc}
$p$, & \quad if $a=1$, \\ [.5em]
$|a|$, & \quad if $a \ne 1$,
\end{tabular}
\right. \]
where $p=\Char(\ff_q)$ and $|a|$ is the order of $a$ in the multiplicative group $\ff_q^\ast$. 
\end{lem}

\begin{proof}
We have $A^{-1}= (\begin{smallmatrix} a & b \\ 0 & 1 \end{smallmatrix})$.
It is easy to see that 
$$\begin{pmatrix} a & b \\ 0 & 1 \end{pmatrix}^m = \begin{pmatrix} a^m & b(a^{m-1}+\cdots +a+1) \\ 0 & 1 \end{pmatrix}.$$
Thus, $(A^{-1})^m = \big( \begin{smallmatrix} 1 & mb \\ 0 & 1 \end{smallmatrix}\big)$ if $a=1$ and 
$(A^{-1})^m = \big(\begin{smallmatrix} a & b \frac{a^m-1}{a-1} \\ 0 & 1 \end{smallmatrix} \big)$ if $a\ne 1$, 
and the result follows. 
\end{proof}

We prove now a technical result which will be used to find a suitable expression of the elements of certain generator matrices of the codes  $\cod(A,\alpha, \infty, r)$.

\begin{lem}\label{alfadiferencia}
Let $A= (\begin{smallmatrix} 1 & -b \\ 0 & a \end{smallmatrix}) \in \pgl_2(\ff_q)$ be a matrix of order $n>1$. Let  $\alpha\in \ff_q$ be such that  $A^{-1}\cdot \alpha\neq \alpha$ and let us consider the orbit $[\alpha]_A$ as in \eqref{orbit}. Then for every $1\leq i<j\leq n$ we have that 
$$\alpha_j-\alpha_{i} = (b + (a-1)\alpha) \alpha^{i-1} \, \sum_{k=0}^{j-i-1}a^k.$$
\end{lem}

\begin{proof}
Notice that $\alpha_1=\alpha$ and for $j=2,\ldots, n$ we have 
$$ \alpha_j=a^{j-1}\alpha+b(1+a+a^2+\cdots + a^{j-2}).$$ 
Therefore, for $2\le i < j \le n$, we have that
$$\alpha_{j} - \alpha_{i} = a^{j-1}\alpha + b \sum_{k=0}^{j-2}{a^{k}} - a^{i-1}\alpha -b\sum_{k=0}^{i-2}{a^{k}} $$
and hence 
\begin{eqnarray*}
\alpha_{j} - \alpha_{i} 
&=&  a^{i-1}\alpha(a^{j-i}-1) +b\sum_{k=i-1}^{j-2}{a^{k}} \\ 
&=& a^{i-1}\alpha(a^{j-i}-1) + b \sum_{k=0}^{j-i-1}{a^{k+i-1}} \\
&=& a^{i-1}\alpha(a^{j-i}-1) + b a^{i-1}\sum_{k=0}^{j-i-1}{a^{k}} \\
&=& a^{i-1} \Big\{ \alpha(a-1)  \sum_{k=0}^{j-i-1}{a^{k}}  + b\sum_{k=0}^{j-i-1}{a^{k}} \Big\} \\
&=& \big( b+(a-1)\alpha \big) a^{i-1}\sum_{k=0}^{j-i-1}{a^{k}}.
\end{eqnarray*}
Similarly, we get 
$\alpha_{j} - \alpha_{1} = (b+(a-1)\alpha) \sum_{k=0}^{j-2}{a^{k}}$ 
for any $j=2,\ldots, n,$ and we are done.
\end{proof}

Recall that  any $[n,k]$-code has  a generator matrix in standard form $M=(I_k \mid W)$ where $I_k$ is the $k\times k$ identity matrix and $W$ is an $k \times (n-k)$ matrix.

Now we show that if 
\[ A = (\begin{smallmatrix} 1&-b \\ 0&a \end{smallmatrix}) \in \pgl_{2}(\ff_{q}),\] the sigma-cyclic code $\cod(A,\alpha,\infty,r)$ is independent of $b$. 

\begin{prop}
Let  $A = (\begin{smallmatrix} 1&-b \\ 0&a \end{smallmatrix}) \in \pgl_{2}(\ff_{q})$ be a matrix of order $n>1$.   Then 
\[\cod(A,\alpha,\infty,r)=\cod(B,\gamma,\infty,r),\] where $B = (\begin{smallmatrix} 1&-b' \\ 0&a \end{smallmatrix}) \in \pgl_{2}(\ff_{q})$.  
\end{prop}

\begin{proof}
Since $G=rP_{\infty}$, from Proposition 2.3.3 of \cite{Sti} we have that a  generator matrix of $\cod(A,\alpha,\infty,r)$ is 
\[ M=\begin{pmatrix}
1            &1           &\ldots&1           &1            &\ldots& 1     \\
\alpha_{1}   &\alpha_{2}  &\ldots&\alpha_{k}  &\alpha_{k+1} &\ldots&\alpha_{n} \\
\alpha_{1}^{2}&\alpha_{2}^{2}&\ldots&\alpha_{k}^{2}&\alpha_{k+1}^{2}&\ldots&\alpha_{n}^{2}\\
\vdots            &      \vdots     & &   \vdots        &     \vdots       & & \vdots     \\
\alpha_{1}^{k-1}&\alpha_{2}^{k-1}&\ldots&\alpha_{k}^{k-1}&\alpha_{k+1}^{k-1}&\ldots&\alpha_{n}^{k-1}\\
\end{pmatrix} = \left(V_{k\times k}  \mid U_{k\times (n-k)}\right). \]
The result will be proved once we check that $M$ is row equivalent to a generator matrix in standard form $(I_k\mid W)$ where the entries of $W$ depend only on $a$.  

We write $V=V_{k\times k}$ and $U=U_{k\times (n-k)}$. Since all the $\alpha_i$'s are different and $V$ is a transposed Vandermonde matrix, we have that 
$$\det(V) = \prod_{j=2}^k \, \prod_{i=1}^{j-1} (\alpha_j-\alpha_i)\neq 0,$$ 
so that  $V$ is an invertible matrix and thus $V^{-1}M = (I_k \mid W)$ where $W=V^{-1}U$.  If we denote by $u_j$ and $w_j$ the $j$-th column of $U$ and $W$ respectively, then $w_j=V^{-1}u_j$ and therefore $Vw_j=u_j$. This linear system has a unique solution given by Cramer's rule: every component $(w_j)_i$ of the column vector $w_j$ satisfy  
$$(w_j)_i=\frac{\det(V_i^j)}{\det(V)},$$ where 
\[
V_{i}^j=
\begin{pmatrix}
1  &\ldots&1      &1      & 1  &\ldots& 1     \\
\alpha_{1}     &\ldots&\alpha_{i-1}  &\alpha_{j}& \alpha_{i+1} &\ldots&\alpha_{k} \\
\alpha_{1}^{2}&\ldots&\alpha_{i-1}^{2}&\alpha_{j}^{2} & \alpha_{i+1}^{2}&\ldots&\alpha_{k}^{2}\\
\vdots            &           & \vdots&     \vdots      &    \vdots        & & \vdots     \\
\alpha_{1}^{k-1}&\ldots&\alpha_{i-1}^{k-1}& \alpha_{j}^{k-1} & \alpha_{i+1}^{k-1}&\ldots&\alpha_{k}^{k-1}\\
\end{pmatrix}
\]
 is the matrix formed by replacing the $i$-th column of $V$ by the column vector $u_j$. Since $V_i^j$ is also a transposed Vandermonde matrix, its determinant is given by
\begin{align*}
\det(V_i^j) = \prod_{s=2}^k \, \prod_{t=1}^{s-1}(\beta_s-\beta_t) =
\prod_{s=1}^{i-1}(\alpha_j-\alpha_s) \prod_{s=i+1}^k (\alpha_s-\alpha_j) \prod_{\underset{s\neq i}{s=2}}^k \;  \prod_{\underset{t\neq i}{t=1}}^{s-1} (\alpha_s-\alpha_t)
\end{align*} 
where $\beta_s=\alpha_s$ if $s\neq i$ and $\beta_i=\alpha_j$. 
 
Thus  
\begin{align*}
(w_j)_i = \frac{\det(V_i^j)}{\det(V)} = \frac{\prod\limits_{s=1}^{i-1}(\alpha_j-\alpha_s) 
	\prod\limits_{s=i+1}^k (\alpha_s-\alpha_j) \prod\limits_{\underset{s\neq i}{s=2}}^k 
	\prod\limits_{\underset{t\neq i}{t=1}}^{s-1} (\alpha_s-\alpha_t)}{\prod\limits_{s=2}^k \prod\limits_{t=1}^{s-1} (\alpha_s-\alpha_t)}, 
\end{align*}
and hence
\begin{align*}
 (w_j)_i& = 
\frac{\prod_{s=1}^{i-1}(\alpha_j-\alpha_s)\prod_{s=i+1}^k
(\alpha_s-\alpha_j) }{\prod_{s=1}^{i-1}(\alpha_i-\alpha_s)\prod_{s=i+1}^k
(\alpha_s-\alpha_i)}
=\prod_{{s=1}}^{i-1}\frac{(\alpha_j-\alpha_s)}{(\alpha_i-\alpha_s)}\prod_{{s=i+1}}^{k}\frac{(-1)(\alpha_j-\alpha_s)}{(\alpha_s-\alpha_i)}.
\end{align*}
From this and Lemma \ref{alfadiferencia} we have
\begin{align*}
(w_j)_i & =(-1)^{k-i}\prod_{{s=1}}^{i-1}\frac{ (a\alpha+b-\alpha)a^{s-1}\sum_{t=0}^{j-s-1}{a^{t}}}{ (a\alpha+b-\alpha)a^{s-1}\sum_{t=0}^{i-s-1}{a^{t}}}
\prod_{{s=i+1}}^{k}\frac{ (a\alpha+b-\alpha)a^{s-1}\sum_{t=0}^{j-s-1}{a^{t}}}{ (a\alpha+b-\alpha)a^{i-1}\sum_{t=0}^{s-i-1}{a^{t}}}\\
 & =(-1)^{k-i}\prod_{s=1}^{i-1}\frac{ \sum_{t=0}^{j-s-1}{a^{t}}}{ \sum_{t=0}^{i-s-1}{a^{t}}}\prod_{{s=i+1}}^{k}a^{s-i}\frac{\sum_{t=0}^{j-s-1}{a^{t}}}{\sum_{t=0}^{s-i-1}{a^{t}}}\\
 & =(-1)^{k-i}\prod_{{s=i+1}}^{k}a^{s-i}\prod_{s=1}^{i-1}\frac{ \sum_{t=0}^{j-s-1}{a^{t}}}{ \sum_{t=0}^{i-s-1}{a^{t}}}\prod_{{s=i+1}}^{k}\frac{\sum_{t=0}^{j-s-1}{a^{t}}}{\sum_{t=0}^{s-i-1}{a^{t}}} ,
\end{align*}
so that
\begin{align*}
(w_j)_i & =(-1)^{k-i}a^{\frac 12 (k-i)(k-i+1)} \prod_{s=1}^{i-1}\frac{ \sum_{t=0}^{j-s-1}{a^{t}}}{ \sum_{t=0}^{i-s-1}{a^{t}}}\prod_{{s=i+1}}^{k}\frac{\sum_{t=0}^{j-s-1}{a^{t}}}{\sum_{t=0}^{s-i-1}{a^{t}}}.
\end{align*}
This shows that each column of $W$ depends only on $a$  as desired.
\end{proof}

The above result states that for any matrix $A = (\begin{smallmatrix} 1 & -b \\ 0 & a \end{smallmatrix}) \in \pgl_2(\ff_q)$ 
of order $n>1$ there are two possibilities: 
\begin{equation*}
\cod(A,\alpha,\infty,r) = \begin{cases}
\cod(D_a,\gamma,\infty,r) &  \qquad \text{if $a \ne 1$, where $D_a = (\begin{smallmatrix} 1 & 0 \\ 0 & a \end{smallmatrix})$,} 
\\[2mm]
\cod(T,\gamma,\infty,r) &  \qquad \text{if $a = 1$, where $T = (\begin{smallmatrix} 1 & 1 \\ 0 & 1 \end{smallmatrix})$.}
\end{cases}
\end{equation*}
In the case $a\neq 1$, we have that $1\in \ff_q$ is not a fixed point of $D_a$ and then
\[\cod(A,\alpha,\infty,r)= \cod(D_a,1,\infty,r).\]    
In the case $a=1$ the only fixed point of $T$ is $\infty$. Therefore
\[ \cod(A,\alpha,\infty,r)=\cod(T,1,\infty,r).\]
Notice that there is only one type of sigma-cyclic rational codes of the form $C_{\rr}(D,rP_\infty)$ of length $p=\mathrm{char}(\ff_q)$: namely those of the form $\cod(T,\alpha,\infty,r)$ because  no element of $\ff_q^{*}$ can have order $p$.

We show now that there is, up to equivalence, only one sigma-cyclic rational code of the form $C_\rr(D,rP_\infty)$ of a given length and dimension.

\begin{prop}\label{equalityinfinity}
Let $c\in\ff_q$ be an element of order $n$ in the multiplicative group $\ff_q^{*}$. 
Let $A = (\begin{smallmatrix} 1 & -b \\ 0 & a \end{smallmatrix}) \in \pgl_2(\ff_q)$ 
of order $n$ and let $\alpha\in \ff_q$ such that $A\cdot\alpha\neq \alpha$. Then we have:
\begin{enumerate}[$(a)$]
	\item If $a\neq 1$ then $\cod(A,\alpha,\infty,r)\sim \cod(D_c,1,\infty,r)$
where $D_c = (\begin{smallmatrix} 1 & 0 \\ 0 & c \end{smallmatrix})$, and \msk 
	
	\item If $a=1$ then $\cod(A,\alpha,\infty,r)= \cod(T,1,\infty,r)$ where $T=(\begin{smallmatrix} 1 & 1 \\ 0 & 1 \end{smallmatrix})$.
\end{enumerate} 
\end{prop}

In other words, given an element  $c\in\ff_q$ of order $n\neq\mathrm{char}(\ff_q)$ in the multiplicative group $\ff_q^{*}$, if $\sigma\in \aut_{\ff_q}(\ff_q(x))$ is of order $n$ such that $\sigma(P_\infty)=P_\infty$ and $\sigma(P_{i \mod n})=P_{i+1 \mod n}$ where $P_1,\ldots,P_n$ are rational places of $\ff_q(x)$, then
	\[C_\rr(P_1+\cdots+P_n,rP_\infty) \sim \cod(D_c,1,\infty,r).\]
Furthermore if $\sigma\in \aut_{\ff_q}(\ff_q(x))$ is of order $p=\mathrm{char}(\ff_q)$ and the above conditions on the places hold for $n=p$ then
	\[C_\rr(P_1+\cdots+P_p,rP_\infty)= \cod(T,1,\infty,r).\]

\begin{proof}
The assertion for the case $a=1$ has already been  proved. In the case $a\neq 1$, it suffices to show that if $a_1$ and $a_2$ are two  elements of order $n$ in the multiplicative group $\ff_q^{*}$ then
\[\cod(D_{a_1},1,\infty,r) \sim \cod(D_{a_2},1,\infty,r).\] 
Let $i=1,2$. Since 
$D_{a_i}^{-1} = (\begin{smallmatrix} a_i & 0 \\ 0 & 1 \end{smallmatrix})$,
a generator matrix for $\cod(D_{a_i},1,\infty,r)$ is of the form
\[
\begin{pmatrix}
1         &1         &1          &\ldots& 1     \\
1         & a_i      &a_i^{2}      &\ldots&  a_i^{(n-1)} \\ 
1         & a^{2}_i    &a_i^{4}      &\ldots&  a_i^{2(n-1)} \\
\vdots    & \vdots   &\vdots     &\vdots& \vdots\\
1         & a^{(k-1)}_i    &a_i^{2(k-1)}      &\ldots&  a_i^{(k-1)(n-1)} 
\end{pmatrix}.
\]
Now the second row is just the cyclic subgroup of $\ff_q^{*}$ generated by $a_i$, that is
\[ \langle a_i\rangle=\{1,a_{i},a_{i}^{2},\ldots , a_{i}^{n-1} \}. \]
Since $a_1$ and $a_2$ have the same order, we have that  $ \langle a_1\rangle= \langle a_2\rangle$ (see Theorem 7 in \cite{DF}).
This means that the generator matrix of $\cod(D_{a_1},1,\infty,r)$ can be obtained from the generator matrix of $\cod(D_{a_2},1,\infty,r)$ by columns permutation.
\end{proof}

We are now in a position to state one of the main results of this paper. 
\begin{thm}\label{mainequality}
There is, up to equivalence, only one sigma-cyclic rational code over $\ff_q$ of a given length and dimension of the form $C_\rr(D,rP_\beta)$, where $D=P_{\alpha_1}+\cdots+P_{\alpha_n}$ as in \eqref{orbit}--\eqref{divD}, $1\le r\le n-2$ and $\beta\in \mathbb{P}^1(\ff_q)$.	
\end{thm}

\begin{proof}
	The result follows inmediately from Propositions \ref{equalitysigma}, \ref{equalityzero} and \ref{equalityinfinity}.
\end{proof}

\section{The sigma-method and cyclic extensions}
Let $F'/F$ be a finite extension of function fields over $\ff_q$. We have seen that the sigma-method   works nicely with elements of $\aut_F(F')$ if $F'/F$ is a cyclic extension. We will show now that we also have a good understanding when the sigma-method is used with elements of $\aut_{\ff_q}(F)$. The main result in this final section is presented in Theorem \ref{ext cic}, where a detailed description of the behavior of the places cyclically moved by an element of   $\aut_{\ff_q}(F)$ is given.

We begin with a result in the slightly more general situation of  working with elements of $\aut_{\ff_q}(F)$. 

\begin{prop}\label{sigmacyclic2}
	Let $F$ be a function field over $\ff_q$ and let $P_1,\ldots,P_n$ be different places of $F$. Suppose there exists $\sigma\in\aut_{\ff_q}(F)$ such that \eqref{conditions} holds. Then there exist a subfield $E$ of $F$ and a place $P$ of $E$ such that $F/E$ is a cyclic extension of degree $m$ divisible by $n$, and $P$ decomposes exactly in $F$ into the places $P_1,\ldots,P_n$ with $e(P_i|P)f(P_i|P)=\tfrac mn$ for $i=1,\ldots,n$. 
\end{prop}

\begin{proof}
Let $G$ be the subgroup of $\aut_{\ff_q}(F)$ generated by $\sigma$ and let $E=F^G$, the fixed field of $G$.
Thus, $F/E$ is a Galois extension with Galois group
	$$\gal(F/E)=G=\langle \sigma \rangle \,.$$
Thus, $F/E$ is a cyclic extension of degree $m$, where $m$ is the order of $\sigma$ in $\aut_{\ff_q}(F)$.
Also $P=P_1\cap E$ is a place of $E$ and, by \eqref{conditions}, we also have that $P=P_i\cap E$ for $i=1,\ldots,n$.
These are, in fact, all the places of $F$ lying above $P$, because $G$ acts transitively on the set of places of $F$ lying above $P$. 
	
We have the fundamental relation 
	$$\sum_{i=1}^n e_i f_i = [F:E]=m,$$ 
where $e_i=e(P_i\,|\, P)$ and $f_i=f(P_i \,|\, P)$, $1\le i\le n$, are the ramification index and the inertia degree, respectively. 	Since $F/E$ is Galois, we have that $e(P_i\,|\,P)=e$ and $f(P_i\,|\,P)=f$, for $i=1,\ldots,n$. 
Thus $ne\!f=m$ and we are done with the first part.
\end{proof}

Note that in the previous proposition, $\ff_q(x)$ is not necessarily contained in $E$. However, it is possible to explicitly construct a subfield $E$ of $F$ containing $\ff_q(x)$  by considering elements of $\aut_{\ff_q}(F)$ in Proposition \ref{sigmacyclic2}. First we need a technical lemma. 

\begin{lem}\label{galoisciclica}
	Consider a cyclic Galois extension $F'/F$ of function fields over $\ff_q$ and let $\tau \in \aut(F')$. Then $F'/\tau(F)$ is also cyclic. 
\end{lem}

\begin{proof}
Let $E=\tau(F)$ and let us show that $F'/E$ is a Galois extension. Since $F'/F$ is a Galois extension, then there is a separable polynomial $f(x)=\sum{a_{i}x^{i}} \in F[x]$ such that $F'$ is a splitting field for $f$.
Then, $$g(x)=\tau(f)(x)=\sum{\tau(a_{i})x^{i}} \in E[x]$$ is separable too.
In fact, if $b \in F'$ is a multiple root of $g$, then $a=\tau^{-1}(b)$ is a multiple root of $f$, which is not possible. 
Moreover, if $a \in F'$ is a root of $f$ then $b=\tau(a) \in F'$ is a root of $g$, and since  $\deg f = \deg g$ we have that $F'$ contains all the roots of $g$.
Thus $F'$ is a splitting field for $g$. In fact, if some proper subfield $K$ of $F'$ contains all the roots of $g$ then $K$ contains all the roots of $f$ contradicting that $F'$ is a splitting field for $f$.
Thus, $F'$ is a splitting field of a separable polynomial $g(x) \in E[x]$ and therefore $F'/E$ is Galois.
	
To prove that $F'/E$ is cyclic we will show that $\gal(F'/E)=\langle \eta \rangle$ 
with $\eta = \tau \circ \rho \circ \tau^{-1}$ where $\gal(F'/F)=\langle \rho \rangle$.
On the one hand, we have that $\eta \in \aut(F')$ and $\eta_{|_E}=id$ because if $z \in E$ then $\tau^{-1}(z) \in F$ and $\rho_{|_F}=id$, thus
	\[\eta(z)=\tau(\rho(\tau^{-1}(z)))=\tau(\tau^{-1}(z))=z.\]
So $\eta \in \gal(F'/E)$, and thus  $\langle \eta \rangle \subset \gal(F'/E)$.
On the other hand, if $\mu \in \gal(F'/E)$ then $$\tau^{-1} \circ \mu \circ \tau \in \gal(F'/F)=\langle \rho \rangle.$$ Therefore,  $\tau^{-1} \circ \mu \circ \tau=\rho^{j}$ for some $j$, and 
$\mu = \tau \circ \rho^{j} \circ \tau^{-1} = \eta^{j} \in \langle \eta \rangle$. 
Thus, we have $\gal(F'/E)=\langle \eta \rangle$, as it was to be shown.
\end{proof}

\goodbreak 
We are now in a position of proving the main result of this section.
\begin{thm} \label{ext cic} 
Let $F$ be a function field over $\ff_{q}$ containing the rational function field $\ff_q(x)$. 
Let $P_{1},\ldots,P_{n} \in \pl(F)$ be different places of $F$ and suppose $\sigma \in \aut_{\ff_q}(F)$ satisfies \eqref{conditions}, that is $\sigma(P_{i \mod n}) = P_{i+1 \mod n}$. Then, we have the following:
	\begin{enumerate}[$(a)$]
		\item There is a function field $E$ such that $\ff_{q}(x) \subset E \subset F$ and $F/E$ is a cyclic Galois extension. \sk 
		
		\item There is a place $S$ of $\ff_{q}(x)$ which splits in $F$ into the places $P_{1},\ldots,P_{n}$. \sk 
		
		\item There are places $Q_{1},\ldots,Q_{k}$ of $E$ which split in $F$ into the places $P_{1},\ldots,P_{n}$.
	\end{enumerate}
\end{thm}

\begin{proof}
In Figure 1 we can see the entire situation we want to prove.
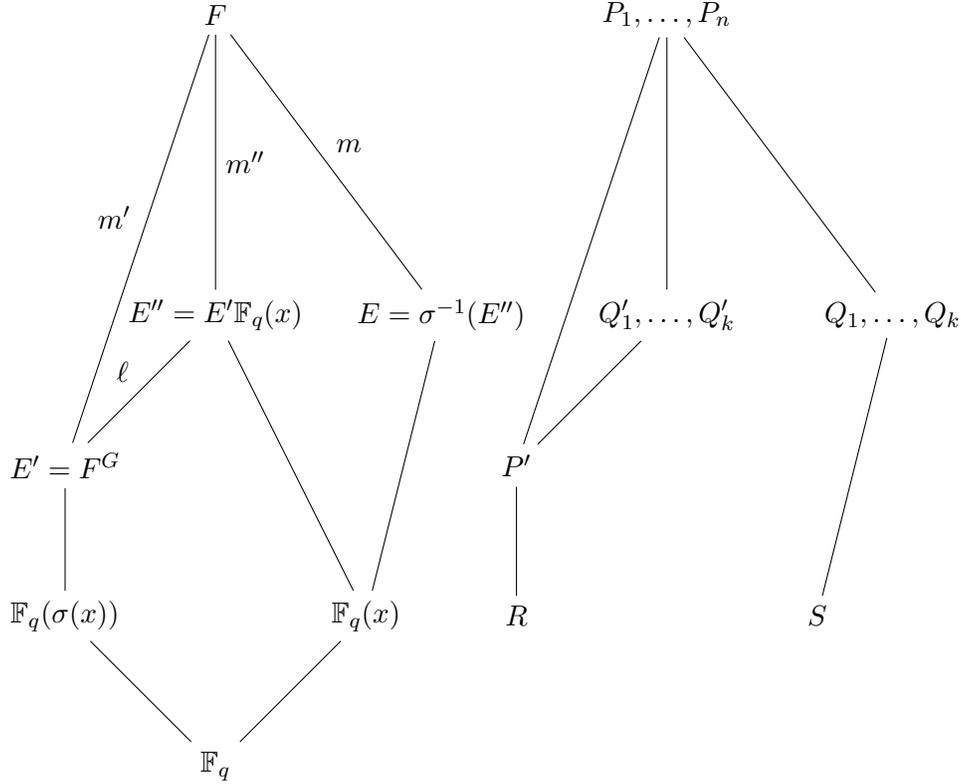
\begin{figure}[H]
		\begin{center} 
			\begin{tikzpicture}[scale=0.7, node distance = 2cm, auto];
			\node (Fq) {$\ff_{q}$};
			\node (Fqsigmax) [above of=Fq, left of=Fq] {$\ff_{q}(\sigma(x))$};
			\node (Fqx) [above of=Fq, right of=Fq] {$\ff_{q}(x)$};
			\node (E') [above of=Fqsigmax] {$E'=F^{G}$};
			\node (E'') [above of=Fq, node distance = 6cm] {$E''=E'\ff_{q}(x)$};
			\node (E) [right of=E'', node distance = 3cm] {$E=\sigma^{-1}(E'')$};
			\node (F) [above of=E'', node distance = 4cm] {$F$};
			\draw[-] (Fq) to node {} (Fqsigmax);
			\draw[-] (Fq) to node {} (Fqx);
			\draw[-] (Fqsigmax) to node {} (E');
			\draw[-] (Fqx) to node {} (E'');
			\draw[-] (Fqx) to node {} (E);
			\draw[-] (E') to node {$m'$} (F);
			\draw[-] (E') to node {$\ell$} (E'');
			\draw[-] (E'') to node [swap]{$m''$} (F);
			\draw[-] (E) to node [swap]{$m$} (F);
			\node (nada) [right of = Fq, node distance = 6 cm] {};
			\node (R) [above of=nada, left of=nada] {$R$};
			\node (S) [above of=nada, right of=nada] {$S$};
			\node (P') [above of=R] {$P'$};
			\node (Q') [above of=nada, node distance = 6cm] {$Q'_{1},\ldots,Q'_{k}$};
			\node (Q) [right of=Q', node distance = 3cm] {$Q_{1},\ldots,Q_{k}$};
			\node (P) [above of=Q', node distance = 4cm] {$P_{1},\ldots,P_{n}$};
			\draw[-] (R) to node {} (P');
			\draw[-] (S) to node {} (Q);
			\draw[-] (P') to node {} (P);
			\draw[-] (P') to node {} (Q');
			\draw[-] (Q') to node {} (P);
			\draw[-] (Q) to node {} (P);
			\end{tikzpicture}
			\caption{Theorem \ref{ext cic} in a picture}
		\end{center}
	\end{figure}

Let $G$ be the subgroup of $\aut_{\ff_q}(F)$ given by $G=\langle \sigma \rangle$ and let $E'=F^G$ be the fixed field of $F$ by $G$. Then $F/E'$ is a cyclic Galois extension with Galois group $G$ of order $m'=|\sigma|$. Furthermore $E'/\ff_{q}$ is a function field over $\ff_q$ and since $\ff_q(x)\subset F$ we then have $\ff_{q}(\sigma(x)) \subset E'$.
	
Consider the place $P'=P_{1}\cap E' \in \pl(E')$. Since $F/E'$ is Galois, $\sigma(P_{i})=P_{i+1}$ for $i=1, \ldots , n-1$ and $\sigma(P_n)=P_1$, then $P_{1},\ldots ,P_{n}$ are all the places of $F$ lying above $P'$ so that for $i=1,\ldots ,n$ we have 
$ e(P_{i}\,|\,P')f(P_{i}\,|\,P')=\frac{m'}{n}$.
	
Now let us consider the composite field $E''=E'\ff_{q}(x)$. Note that $E' \subset E'' \subset F$ and $F/E'$ is a cyclic Galois extension with Galois group $G$. Therefore $F/E''$ is a cyclic Galois extension with Galois group $T=\langle \tau \rangle$, where $\tau = \sigma^{\ell}$ for some $1 \leq \ell \leq m'$, and $[F:E'']=m''=|\tau|$. Furthermore, $\ell m''=m'$, $E''=F^{T}$, and $\ff_{q}(x) \subset E''$. Also $E''/E'$ is a cyclic Galois extension and $[E'':E']=\ell$ (see Corollary 1.11 in \cite{lang}).
	
Note that the places of the form $Q_i'=P_{i}\cap E'' \in \pl(E'')$ are lying above $P'$. In fact,
	\[ Q_i'\cap E'=(P_{i}\cap E'') \cap E' = P_{i} \cap (E''\cap E')=P_i \cap E'=P'.  \]
Since $P_{1}, \ldots , P_{n}$ are all the places of $F$ lying above $P'$, we have that 
$Q_1', \ldots, Q_n'$ are all the places of $E''$ lying above $P'$ (although they may not be pairwise different).
	
Let $Q'_{1}, \ldots , Q'_{k}$ be all the different places of $E''$ lying above $P'$ (hence $k\le n$). Since $E''=F^{T}$, the set of places of $F$ lying above each $Q'_{j}$ is determined by the orbit of the action of $T$ over the set $\{P_{1}, \ldots , P_{n}\}$. Then $P_{i} \mid Q'_{j}$ if and only if $\tau(P_{i}),\tau^{2}(P_{i}),\ldots,\tau^{m'}(P_{i})$ are above $Q'_{j}$.
	
Since $F/E'$, $F/E''$ and $E''/E'$ are cyclic extensions we see that
	\begin{eqnarray*}
		e(P_{i}\,|\,P')f(P_{i}\,|\,P') & = & \frac{m'}{n}, \\
		e(P_{i}\,|\,Q'_{j})f(P_{i}\,|\,Q'_{j}) & = & \frac{m''}{r_{j}}, \\
		e(Q'_{j}\,|\,P')f(Q'_{j}\,|\,P') & = & \frac{\ell}{k}, \\
	\end{eqnarray*}
where $r_{j}$ is the number of places of $F$ lying above $Q'_{j}$. Then
	\[ \frac{m'}{n} = e(P_{i}\,|\,P')f(P_{i}\,|\,P') = e(P_{i}\,|\,Q'_{j})f(P_{i}\,|\,Q'_{j})e(Q'_{j}\,|\,P')f(Q'_{j}\,|\,P') = \frac{m''}{r_{j}} \frac{\ell}{k} = \frac{m'}{r_{j}k} .\]
That is $n=r_jk$. Therefore, for every $1 \leq j \leq k$, we have that there are $r_j=\frac nk$ places of $F$ lying above $Q'_{j}$.

Let $R=P'\cap \ff_{q}(\sigma(x))$. Then $R$ is the place of $\ff_{q}(\sigma(x))$ lying below $P'$ and $S=\sigma^{-1}(R)$ is a place of $\ff_{q}(x)$. Now consider the field 
		$$E= \sigma^{-1}(E'').$$
Since $ R = P' \cap \ff_{q}(\sigma(x)) = Q'_{j}  \cap \ff_{q}(\sigma(x)) $ for any $j$, we then have
	\[S=\sigma^{-1}(R)=\sigma^{-1}(Q'_{j}\cap\ff_{q}(\sigma(x)))=\sigma^{-1}(Q'_{j})\cap\ff_{q}(x)=Q_{j}\cap\ff_{q}(x),\]
with $Q_{j}=\sigma^{-1}(Q'_{j}) \in \pl(E)$. Thus $Q_1,\ldots,Q_k$ are  all the places of $E$ lying above $S$. 
Moreover, each place $Q_j$ splits in $F$ into $\frac nk$ places of the set $\{P_1,\ldots,P_n\}$. This proves ($b$) and ($c$).
	
Finally, Lemma \ref{galoisciclica} implies that $F/E$ is cyclic, since $F/E''$ is cyclic, thus proving $(a)$, and the result follows.
\end{proof}

A particular nice situation occurs when $n$ is a prime number, as we show next.

\begin{coro} If $n=p$ is a prime number in Theorem \ref{ext cic}, then one and only one of following statements holds:
	\begin{enumerate}[$(a)$]
		\item There are exactly $p$ places $Q_{1}, \ldots , Q_{p}$ of $E$ lying below the places $P_1,\ldots,P_p$ of $F$. In other words, each place of $Q_i$ of $E$ given in part (c) of Theorem \ref{ext cic} lies below of only one of the places $P_1,\ldots, P_p$.   
		
		\item There is only one place of $E$ lying below the places $P_{1}, \ldots , P_{p}$. 
	\end{enumerate}
\end{coro}

\bibliographystyle{plain}
\bibliography{referenciasAG}
\end{document}